\documentclass[11pt,reqno,twoside]{article}

\usepackage{fixltx2e} 

\usepackage{cmap} 

\usepackage[T1]{fontenc}
\usepackage[utf8]{inputenc}
\usepackage{graphicx}
\usepackage{placeins}
\usepackage{enumerate}
\usepackage{algcompatible}

\usepackage{subcaption}

\usepackage{verbatim}

\usepackage{setspace}

\let\counterwithin\relax  
\usepackage{lmodern} 
\usepackage[scale=0.88]{tgheros} 

\usepackage{bm} 

\usepackage{amsmath,amsbsy,amsgen,amscd,amsthm,amsfonts,amssymb} 

\usepackage[centering,top=1.6in,bottom=1.5in,left=1in,right=1in]{geometry}

\usepackage{titling}

\graphicspath{ {./images/} }

\usepackage[sf,bf,compact]{titlesec}

\usepackage{booktabs,longtable,tabu} 
\usepackage[font=small,margin=12pt,labelfont={sf,bf},labelsep={space}]{caption}

\usepackage[usenames,dvipsnames]{xcolor}

\definecolor{dark-gray}{gray}{0.3}
\definecolor{dkgray}{rgb}{.4,.4,.4}
\definecolor{dkblue}{rgb}{0,0,.5}
\definecolor{medblue}{rgb}{0,0,.75}
\definecolor{rust}{rgb}{0.5,0.1,0.1}

\usepackage{url}
\usepackage[colorlinks=true]{hyperref}
\hypersetup{linkcolor=dkblue}    
\hypersetup{citecolor=rust}      
\hypersetup{urlcolor=rust}     

\usepackage[final]{microtype} 

\newtheoremstyle{myThm} 
    {\topsep}                    
    {\topsep}                    
    {\itshape}                   
    {}                           
    {\sffamily\bfseries}                   
    {.}                          
    {.5em}                       
    {}  

\newtheoremstyle{myRem} 
    {\topsep}                    
    {\topsep}                    
    {}                   
    {}                           
    {\sffamily}                   
    {.}                          
    {.5em}                       
    {}  

\newtheoremstyle{myDef} 
    {\topsep}                    
    {\topsep}                    
    {}                   
    {}                           
    {\sffamily\bfseries}                   
    {.}                          
    {.5em}                       
    {}  

\theoremstyle{myThm}
\newtheorem{theorem}{Theorem}[section]
\newtheorem{lemma}[theorem]{Lemma}
\newtheorem{proposition}[theorem]{Proposition}
\newtheorem{corollary}[theorem]{Corollary}

\theoremstyle{myRem}
\newtheorem{remark}[theorem]{Remark}

\theoremstyle{myDef}

\usepackage{fancyhdr}
\usepackage{nopageno} 
\usepackage{algorithm}
\usepackage{algpseudocode}

\let\originalleft\left
\let\originalright\right
\renewcommand{\left}{\mathopen{}\mathclose\bgroup\originalleft}
\renewcommand{\right}{\aftergroup\egroup\originalright}

\usepackage{mathtools}
\mathtoolsset{centercolon}  

\renewcommand{\phi}{\varphi}
\newcommand{\eps}{\varepsilon}

\renewcommand{\L}{\mathcal{L}}

\definecolor{mygreen}{rgb}{0.1,0.75,0.2}

\newcommand{\Y}{\mathcal{Y}}
\newcommand{\X}{\mathcal{X}}

\newcommand{\M}{\mathcal{M}}

\newcommand{\prior}{\Pi_n}  
\newcommand{\Pim}{\Pi_n^\M}  
\newcommand{\pim}{\pi_n^\M}  

\newcommand{\wn}{W_n}  
\newcommand{\Wn}{W_n^\M}  

\newcommand{\dkl}{d_{\mbox {\tiny{\rm KL}}}}

\newcommand{\Nc}{\mathcal{N}}

\usepackage[]{algorithm}

\usepackage{graphicx}

\usepackage{soul}
\usepackage{authblk}
\usepackage[square,numbers]{natbib}
\usepackage{chngcntr}
\usepackage{mathrsfs}
\counterwithin{table}{section}
\counterwithin{algorithm}{section}

\title{Unlabeled Data Help in Graph-Based Semi-Supervised Learning: \\ A Bayesian Nonparametrics Perspective} 
\author{Daniel Sanz-Alonso and Ruiyi Yang}

\vspace{.25in} 

\date{University of Chicago}

\makeatletter\@addtoreset{section}{part}\makeatother%
\numberwithin{equation}{section}

\newcommand{\upperRomannumeral}[1]{\uppercase\expandafter{\romannumeral#1}}

\renewcommand{\hat}{\widehat}

\begin{document}
\maketitle 

\begin{abstract}
In this paper we analyze the graph-based approach to semi-supervised learning under a manifold assumption. We adopt a Bayesian perspective and demonstrate that, for a suitable choice of prior constructed with sufficiently many unlabeled data, the posterior contracts around the truth at a rate that is minimax optimal  up to a logarithmic factor. Our theory covers both regression and classification. 
\end{abstract}

\section{Introduction}
Semi-supervised learning (SSL) refers to machine learning techniques that combine during training a small amount of labeled data with a large amount of unlabeled data. SSL has received increasing attention over the past decades because in many recent applications labeled data are expensive to collect while unlabeled data are abundant. Examples include analysis of body-worn videos and video surveillance \citep{qiao2019uncertainty}, text categorization and translation \citep{shi2010cross}, image classification \citep{zhu2005semi} and protein structure prediction \citep{weston2005semi}.  The aim of this paper is to investigate whether unlabeled data can enhance learning performance. The answer to such question necessarily depends on the model assumptions and the methodology employed. Our main contribution is to show that under a standard manifold assumption, unlabeled data are helpful when using graph-based methods in a Bayesian setting. We do so by establishing that the optimal posterior contraction rate is achieved (up to a logarithmic factor) provided that the size of the unlabeled dataset grows sufficiently fast with the size of the labeled dataset. 

The SSL problem of interest can be informally described as follows. Given labeled data $\{ (X_1,Y_1),\ldots,(X_n,Y_n) \}$ and unlabeled data $\{X_{n+1},\ldots,X_{N_n}\},$ the goal is to predict $Y$ from $X.$ More precisely, we are interested in the  SSL problem of inferring the regression function $f_0(x) :=\mathbb{E} (Y|X=x)$ at the given features $\X_{N_n} := \{X_1, \ldots, X_{N_n} \}$ in either of these two settings:
\begin{enumerate}
    \item Regression: $Y=f_0(X)+\eta$, where $\eta \sim \mathcal{N}(0,\sigma^2)$ with $\sigma^2$ known.
    \item  Binary classification: $\mathbb{P}(Y=1|X)=f_0(X)$. 
\end{enumerate}

We analyze graph-based methods applied to this inference task under the \emph{manifold assumption}  that $X$ takes values on a hidden manifold $\M$.  This assumption is natural when the features $\X_{N_n}$ live in a high-dimensional ambient space but admit a low-dimensional representation \citep{bickel2007local,niyogi2013manifold,trillos2019local,trillos2020consistency}.  For instance,  \cite{hein2005intrinsic,costa2006determining} demonstrate the low intrinsic dimension of standard datasets for image classification. Graph-based methods are well-suited under the manifold assumption as they allow to uncover the geometry of the hidden manifold and promote smoothness of the inferred $f_0$ along it. Indeed,  graph-based methods are among the most powerful classification techniques for SSL problems where similar features are expected to belong to the same class \citep{belkin2004regularization,belkin2004semi,zhu2005semi}. The central idea behind traditional graph-based methods is to  infer $f_0$ by minimizing an objective function comprising at least two terms:  (i) a regularization term constructed with a graph-Laplacian of the features $\mathcal{X}_{N_n},$ which leverages the ability of unsupervised graph-based techniques to extract geometric information from $\M$;  and (ii) a data-fidelity term that incorporates the labeled data. We adopt an analogous Bayesian perspective where: (i) the prior distribution $\prior(\cdot \,|\, \X_{N_n} )$ will be defined using a graph-Laplacian of the features $\X_{N_n}$  (hence the notation ``given $\X_{N_n}$'')  to extract geometric information from $\M$; and (ii) the likelihood function promotes matching the labeled data. Assuming the data are independent so that the likelihood factorizes, the posterior takes the form 
\begin{align}
    \prior(B \,| \, \X_{N_n},\Y_n)  \propto \int_B \prod_{i=1}^n L_{Y_i|X_i}(f) \,d\prior(f \,|\, \X_{N_n} ), \quad \quad  B\in \mathcal{B},  \label{eq:discrete posterior}
\end{align}
where $\Y_n := \{ Y_1, \ldots, Y_n\}$ and $\mathcal{B}$ is the Borel $\sigma$-algebra on $\mathbb{R}^{N_n}$ (here we are identifying functions over $\mathcal{X}_{N_n}$ with $\mathbb{R}^{N_n}$).
The conditional likelihood of $Y_i|X_i$ is given by
\begin{align}
    L_{Y_i|X_i}(f) =\frac{1}{\sqrt{2\pi\sigma^2}}\exp\left(-\frac{|Y_i-f(X_i)|^2}{2\sigma^2}\right) \label{eq:discrete likelihood 1}
\end{align}
in the regression setting and 
\begin{align}
    L_{Y_i|X_i}(f) =f(X_i)^{Y_i} \left[1-f(X_i)\right]^{1-Y_i}\label{eq:discrete likelihood 2}
\end{align}
in the classification setting.

We shall analyze our Bayesian approach in a frequentist perspective by assuming the data is generated from a fixed truth $f_0$ and studying the contraction rate of the posterior around $f_0$, defined as the sequence of real numbers $\varepsilon_n$ such that 
\begin{align*}
    \mathbb{E}_{X}\mathbb{E}_{f_0} \Pi_n\left(f: d_n(f,f_0)\geq M_n\varepsilon_n\,|\,\mathcal{X}_{N_n},\mathcal{Y}_{n}\right)\xrightarrow{n\rightarrow \infty} 0, 
\end{align*}
for any sequence $M_n\rightarrow \infty$ and some suitable semi-metric $d_n$. Here the double expectation $\mathbb{E}_X\mathbb{E}_{f_0}$ is taken first with respect to the conditional data distribution of $\mathcal{Y}_n|X_1,\ldots,X_n$ specified by $f_0$ and then with respect to the marginal of $\mathcal{X}_{N_n}$. The idea of posterior contraction rate was formally introduced in \cite{ghosal2000convergence,shen2001rates} and has since then become a popular criterion for analyzing Bayesian methods. In particular, contraction of the posterior with rate $\varepsilon_n$ gives a point estimator defined as 
\begin{align*}
    \widehat{f}_n=\underset{g}{\operatorname{arg\,max}}\,\Big[ \Pi_n(f:d_n(f,g)\leq M_n\varepsilon_n\,|\, \mathcal{X}_{N_n},\mathcal{Y}_n) \Big]
\end{align*}
that converges to $f_0$ with the same rate $\varepsilon_n$, which together with the minimax theory for statistical estimation quantifies the performance of our Bayesian SSL procedure.

Our main result, Theorem \ref{thm:Main Thm}, shows that the posterior \eqref{eq:discrete posterior} contracts around $f_0$  at the minimax optimal rate (up to logarithmic factors), provided that the prior distribution $\prior(\cdot \,|\, \X_{N_n} )$ is carefully designed and that $N_n$ grows at a certain polynomial rate with $n$. More broadly, our theory suggests that graph-based methods require abundant unlabeled data in order to effectively extract geometric information to regularize SSL problems.

\subsection{Related Work}
The question of whether unlabeled data enhance SSL performance has been widely studied.  Positive and negative conclusions have been reached depending on the assumptions made on the relationship between the target function and the marginal data distribution, and also on the methodology employed \citep{cozman2006risks,bickel2007local,liang2007use,wasserman2008statistical,niyogi2013manifold,singh2008unlabeled}. Two standard model assumptions  are the \emph{clustering assumption} \citep{seeger2000learning}, which states that the target function is smooth over high density regions,  and the \emph{manifold assumption} that we have introduced. The latter has been extensively used in the statistical learning literature, and specifically in SSL, e.g. \cite{bickel2007local,wasserman2008statistical,castillo2014thomas,yang2016bayesian,trillos2019local,trillos2020consistency}. 

 Several methodologies for SSL have been developed based on generative modeling, support vector machines, semi-definite programming, graph-based methods, etc. (see the overview in \cite{zhu2005semi,chapelle2006semi}). Our focus is on  graph-based approaches that combine label information with geometric information extracted from the unlabeled data employing graphical unsupervised techniques \citep{von2007tutorial}.  This heuristic has motivated the use of graph-based regularizations in a wide number of applications, but a rigorous analysis of the mechanisms by which unlabeled data enhance the performance of graph-based SSL methods is only starting to emerge. The recent papers \cite{bertozzi2020posterior,hoffmann2020consistency} studied posterior \emph{consistency} (the asymptotics of the posterior probability $\Pi_n(f:d_n(f,f_0)\geq\varepsilon\,|\, \operatorname{data})$ for a fixed small number $\varepsilon$)  for a fixed sample size in the small noise limit, whereas we consider the large $n$ limit and further establish posterior contraction rates. Rates of convergence for optimization rather than Bayesian formulations of graph-based SSL have been recently established in  \cite{calder2020rates}.   In a Bayesian framework, \cite{garcia2018continuum,trillos2020consistency} show the continuum limit of posterior distributions as the size of the unlabeled dataset grows, without increasing the size of the labeled dataset. These papers did not investigate whether the posteriors contract around the truth, and they did not demonstrate the value of unlabeled data in boosting learning performance. The work \cite{kirichenko2017estimating} studies fully-supervised function estimation on large graphs assuming that the truth changes with the size of the graph. In contrast, we investigate posterior contraction in a SSL setting with a \emph{fixed} truth $f_0$ defined on the underlying manifold $\M$.

Our analysis uses tools from Bayesian nonparametrics and spectral analysis of graph-Laplacians. We will provide the necessary background on posterior contraction in Section \ref{sec:background pc}, and  we refer  to \cite{ghosal2017fundamentals}[Chapters 6, 8 and 11] for a comprehensive introduction to this subject. While numerous results on spectral convergence of graph-Laplacians can be found in the literature, our analysis of posterior contraction requires bounds in $L^\infty$-metric with rates, recently developed in \cite{dunson2019spectral,calder2020lipschitz}.


\subsection{Main Contributions and Scope}\label{sec:scope}

Our main result, Theorem \ref{thm:Main Thm}, is to our knowledge the first to establish posterior contraction rates for graph-based SSL. In doing so, we provide novel understanding on the relative value of labeled and unlabeled data, and we set forth a rigorous quantitative framework in which to analyze this question. We point out, however, two caveats. First, our theory is non-adaptive in the sense that the Bayesian methodology we analyze only achieves optimal contraction rates when a priori information on the smoothness of the truth is available. Our analysis is motivated by existing graph-based techniques, and the development and analysis of adaptive graph-based methods for SSL is an important research direction stemming from our work, but beyond the scope of this paper. Second, our results build on existing spectral convergence rates of graph-Laplacians that may be suboptimal; as a consequence,  the required sample size of the unlabeled data  that we establish may not be sharp. Due to the plug-in nature of our analysis, improved spectral convergence rates will automatically translate into a sharper bound on the sample complexity. Despite these caveats, our theory provides evidence that letting the unlabeled dataset  grow polynomially with the labeled dataset, as suggested by Theorem \ref{thm:Main Thm}, is a fundamental requirement for standard graph-based methods to achieve optimal contraction. 

A central part of our proof is devoted to analyzing the convergence of a discretely indexed Gauss-Markov random field in an unstructured data cloud to a Mat\'ern-type Gaussian field on $\M$. This is formalized in Theorem \ref{thm:L infty conv}, a result that we believe to be of independent interest. Finally, our work contributes to the Bayesian nonparametrics literature on manifolds \citep{castillo2014thomas}.

\paragraph{Outline}
The rest of this paper is organized as follows. Section \ref{sec:main result} introduces the construction of the graph-based prior and states our main result. Section \ref{sec:background pc} provides the necessary background on posterior contraction and outlines our analysis. 
Section \ref{sec:analysis} proves our main result, and we close in Section \ref{sec:discussion}.

\paragraph{Notation} We denote by $\mathcal{L}(Z)$ the law of the random variable $Z$. For $a,b$ two real numbers, we denote $a\wedge b=$ min$\{a,b\}$ and $a\vee b=$ max$\{a,b\}$. The symbol $\lesssim$ will denote less than or equal to up to a universal constant. For two real sequences $\{a_n\}$ and $\{b_n\}$, we denote (i) $a_n\ll b_n$ if $\operatorname{lim}_n (b_n/a_n)=0$; (ii) $a_n=O(b_n)$ if $\operatorname{lim\, sup}_n (b_n/a_n)\leq C$ for some positive constant $C$; and  (iii) $a_n\asymp b_n$ if  $c_1\leq \operatorname{lim\,inf}_n (a_n/b_n) \leq \operatorname{lim\,sup}_n (a_n/b_n) \leq c_2$ for some positive constants $c_1,c_2$.  Finally we let $\gamma$ denote the Lebesgue measure on $\mathbb{R}$.

\section{Prior Construction and Main Result}\label{sec:main result}
In this section we introduce the graph-based prior $\prior (\cdot \, |\, \X_{N_n})$ and we state the main result of this paper. Before doing so, we formalize our setting. 
We assume to be given labeled data $(X_1,Y_1),\ldots,(X_n,Y_n)\overset{i.i.d.}\sim \L(X,Y)$ and unlabeled data $X_{n+1},\ldots,X_{N_n} \overset{i.i.d.}\sim \L(X) $, where $\{X_i\}_{i=n+1}^{N_n}$ are independent from $\{X_i\}_{i=1}^n$. 
Recall that the goal is to estimate the regression function
 \begin{align*}
    f_0(x) =\mathbb{E}(Y|X=x)
\end{align*}
at the given features $\X_{N_n} = \{X_1, \ldots, X_{N_n} \}.$ 
We suppose that $\mu : = \L(X)$ is supported on an $m$-dimensional smooth, connected, compact manifold $\mathcal{M}$ without boundary embedded in $\mathbb{R}^d$, with the absolute value of the sectional curvature bounded and with Riemannian metric inherited from the embedding.  For technical reasons, we further assume that $m\geq 2$ and that $\M$ is a homogeneous manifold (the group of isometries acts transitively on $\M$, e.g. spheres and tori) normalized so that its volume is $1$. We assume for simplicity that $\mu$ is the uniform distribution on $\M.$  As discussed in Section \ref{sec:discussion},  our results can be generalized  to nonuniform marginal density.

\subsection{Graph-based Prior}\label{sec:graph prior}
We now describe the construction of the graph-based prior  $\prior (\cdot \, |\, \X_{N_n} )$ on $f_0$ restricted to the given features $\X_{N_n}.$ The priors we consider have the general form  
\begin{align}\label{eq:generalprior}
    \prior (\cdot \, |\, \X_{N_n})=\mathcal{L}(\Phi(\wn )\,|\, \X_{N_n} ),
\end{align}
where $\wn$ is a Gauss-Markov random field in $\mathbb{R}^{N_n}$ whose covariance will be defined in terms of  a graph-Laplacian \citep{von2007tutorial} of $\X_{N_n}.$  For the regression problem, $\Phi$ is taken to be the identity. For the classification problem, where $f_0$ takes values in $(0,1),$  $\Phi:\mathbb{R}\rightarrow (0,1)$ is a link function, which we assume throughout to be invertible and twice differentiable with $\Phi^{\prime}/(\Phi(1-\Phi))$ uniformly bounded and $\int (\Phi^{\prime\prime})^2/\Phi^{\prime}\, d\gamma <\infty$. The logistic function, for instance, satisfies all these standard requirements.

In the remainder of this subsection we introduce and motivate our construction of the Gauss-Markov random field $\wn.$ The starting point is to define a similarity matrix $H\in\mathbb{R}^{N_n\times N_n}$ whose entries $H_{ij} \ge 0$ represent the closeness between features $X_i$ and $X_j$. For reasons discussed in Subsection \ref{ssec:interpretationmatern} below, we set
\begin{align}
    H_{ij} :=\frac{2(m+2)}{N_n\nu_m\zeta_{N_n}^{m+2}}\mathbf{1} \bigl\{|X_i-X_j|<\zeta_{N_n} \bigr\}, \label{eq:similarity}
\end{align}
where $|\cdot|$ is the Euclidean distance in $\mathbb{R}^d$, $\nu_m$ is the volume of the $m-$dimensional unit ball, and $\zeta_{N_n}$ is the {\em connectivity} of the graph, a user-specified parameter. Recall that a graph-Laplacian is a positive semi-definite matrix obtained by suitably transforming a similarity matrix. For concreteness, we work with the unnormalized graph-Laplacian matrix $\Delta_{N_n}:=D-H$, where $D$ is the diagonal matrix with entries $D_{ii}=\sum_{j=1}^{N_n}H_{ij}$. For a vector $v\in\mathbb{R}^{N_n}$,  naturally identified with a function on the data cloud $\X_{N_n}$, we have 
\begin{align}
    v^T\Delta_{N_n}v=\frac12\sum_{i=1}^{N_n}H_{ij}|v(X_i)-v(X_j)|^2. \label{eq:GL}
\end{align}
We see that indeed $\Delta_{N_n}$ is positive semi-definite and that functions $v$ that change slowly with respect to the similarity $H$ have a small value of $v^T\Delta_{N_n}v.$ This naturally suggests considering the Gaussian distribution $\mathcal{N}(0,(I_{N_n}+\Delta_{N_n})^{-s})$ since the likelihood for the samples will then have a term similar to \eqref{eq:GL}. Based on this idea we shall define the Gauss-Markov random field  $\wn$ as the principal components of it. Precisely, let $\{(\lambda_i^{(N_n)},\psi_i^{(N_n)})\}_{i=1}^{N_n}$ be the ordered eigenpairs of $\Delta_{N_n}$ and define 
\begin{align}
    \wn =\sum_{i=1}^{k_{N_n}} \left[1+\lambda_i^{(N_n)}\right]^{-\frac{s}{2}} \xi_i \psi_i^{(N_n)}, \quad \quad \xi_i\overset{i.i.d.}{\sim} \mathcal{N}(0,1), \label{eq:KLtrun}
\end{align}
where $k_{N_n}\ll N_n$ is to be determined. 
Notice that $W_n$ as in \eqref{eq:KLtrun} is a truncation of the Karhunen-Lo\`eve expansion of samples from  $\mathcal{N}(0,(I_{N_n}+\Delta_{N_n})^{-s})$ and therefore the terminology ``principal components''.  

The law of $\wn$ depends  on three parameters: the graph {\em connectivity} $\zeta_{N_n}$ used to define the similarity matrix $H,$ the {\em smoothness} parameter $s,$ and the principal components {\em truncation} parameter $k_{N_n}.$ The connectivity $\zeta_{N_n}$ determines the sparsity of the precision matrix of $W_n$, and it  should be taken to decrease with $N_n$ to better resolve the geometry of $\M$ as more unlabeled data are available. The smoothness $s$ controls the level of regularization.  Larger $s$ leads to  faster decay of the coefficients in \eqref{eq:KLtrun} and smoother samples. The truncation parameter $k_{N_n} \ll N_n$ allows us to keep only the components of the graph-Laplacian that contain useful geometric information on $\M.$ Suitable choices and scalings of these three parameters, as well as further insights on their interpretation, will be given  as we develop our theory. Importantly, the construction does not require knowledge of the underlying manifold $\M$, but only of its dimension. These data-driven Gauss-Markov fields have been used within various intrinsic approaches to Bayesian learning, see e.g. \cite{garcia2018continuum,trillos2019local,trillos2020consistency,harlim2020kernel}.


\subsubsection{Interpretation as a Discretely Indexed Mat\'ern Field}\label{ssec:interpretationmatern}
The Gauss-Markov field $\wn$ can be interpreted as a discretely indexed Mat\'ern Gaussian field  \citep{sanz2020spde}. Consider the Gaussian measure $\Nc(0,  (I-\Delta)^{-s}) ),$ where $-\Delta$ denotes the Laplace-Beltrami operator on $\M$ and the fractional order operator is defined spectrally. Then,  draws from $\Nc(0,  (I-\Delta)^{-s}) )$   admit a representation 
\begin{align}
    W^{\M}=\sum_{i=1}^{\infty}(1+\lambda_i)^{-\frac{s}{2}}\xi_i \psi_i,\quad \quad \xi_i\overset{i.i.d.}{\sim} \mathcal{N}(0,1),\label{eq:continuumprior}
\end{align}
where $(\lambda_i,\psi_i)$'s are the ordered eigenpairs of the Laplace-Beltrami operator. Note the analogy with \eqref{eq:KLtrun}. The field $W^{\M}$ is a generalization of the Mat\'ern model to compact manifolds  \citep{lindgren2011explicit,sanz2020spde}. 
An important step in our analysis of posterior contraction will be to show the convergence of $\wn$ towards $W^{\M},$ in a sense to be made precise in Subsection \ref{sec:L infty conv}, provided that the connectivity, smoothness and truncation parameters of $\wn$ are suitably chosen.
 The similarity matrix defined in \eqref{eq:similarity} enables this convergence result, but other choices (e.g. $K$-nearest neighbors) are possible. Key to showing convergence of $\wn$ to $W^{\M}$  is the spectral convergence of $\Delta_{N_n}$ to $-\Delta$. The truncation parameter $k_{N_n}$ is motivated by the fact that only the first eigenpairs of $\Delta_{N_n}$ accurately approximate those of $-\Delta;$ see Proposition  \ref{thm:specRate} below for a precise statement.





\subsubsection{Reconciling Optimization and Bayesian Perspectives}
To further motivate the definition of $\wn$, we relate our prior construction to the optimization literature. To streamline the discussion, let us focus on the regression problem where the link function $\Phi$ in \eqref{eq:generalprior} is the identity map. Classical graph-based optimization recovers $f_0$ at $\X_{N_n}$ using $f^T\Delta_{N_n}f$ as a regularizer and an appropriate data-fidelity term to match the labeled data,  for instance,
\begin{align}
    \hat{f}_0:=\underset{f\in\mathbb{R}^{N_n}}{\operatorname{arg\,min}}\, \sum_{i=1}^n |Y_i-f(X_i)|^2 +\lambda f^T  \Delta_{N_n}f, \label{eq:opt}
\end{align}
where $\lambda$ controls the level of regularization. The solution to \eqref{eq:opt} is conceptually equivalent to the \emph{maximum a posteriori} estimator in the Bayesian approach when the prior is chosen to be $\mathcal{N}(0, \Delta_{N_n}^{-1})$, where $\Delta_{N_n}^{-1}$ denotes the pseudo-inverse of $\Delta_{N_n}.$ 
The paper \cite{nadler2009semi}  shows that \eqref{eq:opt} is not well-posed  when $m\geq 2,$ in the sense that as the number of unlabeled data points increases the solution degenerates to a noninformative function. The authors suggest that this issue can be alleviated by defining the regularization term as $\lambda f^T\Delta_{N_n}^s f $ with $s > \frac{m}{2},$ so that higher order ``derivatives'' of $f$ are controlled.  
Similar behavior has been observed with $p$-Laplacian regularizations \citep{el2016asymptotic}. The parameter $s$ in \eqref{eq:KLtrun} plays the exact same role, and we shall see that suitably scaling $s$ with $m$ is needed to warrant consistent learning in the limit of large unlabeled datasets.

\subsection{Main Result}
Now we are ready to state our main result. Let $\delta>0$ be arbitrary. Let $p_m = \frac34$ if $m=2$ and $p_m=\frac{1}{m}$ otherwise. Let $\alpha_s=\frac{6m+6}{2s-3m+1}$ if $s\leq \frac{9}{2}m+\frac52$ and $\alpha_s=1$ otherwise. 
\begin{theorem}\label{thm:Main Thm}
Suppose $\Phi^{-1}(f_0)\in B_{\infty,\infty}^{\beta}$. Let $\prior$ be the prior defined by \eqref{eq:generalprior} and \eqref{eq:KLtrun} with $s>\frac{3}{2}m-\frac12$. Consider the following scaling for $\zeta_{N_n},k_{N_n}$ and $N_n$. 
\begin{enumerate}
    \item $m\leq 4$: $\zeta_{N_n}\asymp N_n^{-\frac{1}{m+4+\delta}}(\log N_n)^{\frac{p_m}{2}},k_{N_n}\asymp N_n^{\frac{m}{(m+4+\delta)(2s-3m+1)\alpha_s}}(\log N_n)^{-\frac{mp_m}{(4s-6m+2)\alpha_s}}$ with
    \begin{align}
        N_n\asymp n^{(m+4+\delta)\alpha_s} (\log n)^{\frac{(m+4+\delta)p_m}{2}}. \label{eq:scaling N_n 1}
    \end{align}
    \item $m\geq 5$: $\zeta_{N_n}\asymp N_n^{-\frac{1}{2m}}(\log N)^{\frac{p_m}{2}},k_{N_n}\asymp N_n^{\frac{1}{(4s-6m+2)\alpha_s}}(\log N_n)^{-\frac{mp_m}{(4s-6m+2)\alpha_s}}$ with
    \begin{align}
        N_n\asymp n^{2m\alpha_s}(\log n)^{mp_m}. \label{eq:scaling N_n 2}
    \end{align}
\end{enumerate}
Then, for $\eps_n$ a multiple of $n^{-\frac{(s-\frac{m}{2})\wedge \beta}{2s}}(\log n)^{\frac{(s-\frac{m}{2})\wedge \beta}{4s-2m}}$ and all $M_n\rightarrow \infty$,
\begin{align}
    \mathbb{E}_X\mathbb{E}_{f_0} \prior\left(f\in L^{\infty}(\mu_{N_n}): \|f-f_0\|_n\geq M_n\eps_n\,|\,  \X_{N_n},\Y_n \right)\xrightarrow{n\rightarrow \infty}0, \label{eq:pc}
\end{align}
where $\mu_{N_n}$ is the empirical measure of $\X_{N_n}$ and  $\|f-f_0\|^2_n:=\frac{1}{n}\sum_{i=1}^n |f(X_i)-f_0(X_i)|^2$. Here $\mathbb{E}_{f_0}$ denotes expectation with respect to the conditional distribution of $\mathcal{Y}_n|X_1,\ldots,X_n$ determined by $f_0$ and $\mathbb{E}_X$ denotes expectation with respect to the marginal distribution of $\mathcal{X}_{N_n}$.  
\end{theorem}

Theorem \ref{thm:Main Thm} presents the posterior contraction rates with respect to the priors constructed in Section \ref{sec:graph prior} under suitable choices of the parameters. Several remarks are in order.


The truth $f_0$ is assumed to belong to a Besov-type space $B_{\infty,\infty}^{\beta}$ on manifolds, defined in Subsection \ref{sec:continuum} in analogy to the usual Besov space $B_{\infty,\infty}^{\beta}(\mathbb{R}^m)$ on Euclidean space. We recall that $B_{\infty,\infty}^{\beta}(\mathbb{R}^m)$ coincides with the  H\"older space $\mathcal{C}^{\beta}(\mathbb{R}^m)$ (the space of functions with $\lfloor \beta\rfloor$ continuous derivatives whose $\lfloor \beta\rfloor$-th derivative is $\beta-\lfloor \beta\rfloor$-H\"older continuous) when $\beta\notin \mathbb{N}$ and contains $\mathcal{C}^{\beta}(\mathbb{R}^m)$ when $\beta\in\mathbb{N}$ (see e.g. \cite{triebel1992theory}[Theorem 1.3.4]). Therefore,  $B_{\infty,\infty}^{\beta}$ can be interpreted as an analog of the H\"older space with parameter $\beta$ on manifolds, and in particular represents a space of $\beta$-regular functions. 

The key implication of Theorem \ref{thm:Main Thm} is that when $s-\frac{m}{2}=\beta$, we recover the rate $n^{-\beta/(2\beta+m)}(\log n)^{1/4}$, which is minimax optimal up to a logarithmic factor for $\beta$-regular functions. The assumption $s>\frac{3}{2}m-\frac{1}{2}$ then requires $\beta>m-\frac{1}{2}$, so optimal contraction rates can only be attained if $f_0$ is not too rough. Theorem \ref{thm:Main Thm} can be extended to hold for all $s>m$ if the eigenfunctions of the Laplace-Beltrami operator on $\M$ are uniformly bounded, which holds for the family of flat manifolds \citep{toth2002riemannian} that include for instance the tori. As mentioned in Subsection \ref{sec:scope}, the above choice of $s$  requires knowing the regularity of $f_0$ and is not adaptive. 

Another key ingredient of the result is the scaling for $N_n$ as in \eqref{eq:scaling N_n 1} and \eqref{eq:scaling N_n 2}, which are both larger than a multiple of $n^{2m}$ since $\alpha_s\geq 1$ in all cases. In other words, the required sample size of the unlabeled data should grow polynomially with respect to the sample size of the  labeled  data  in order to achieve the near optimal contraction rates described above, thereby justifying the claim that unlabeled data help.  

We further remark that Theorem \ref{thm:Main Thm} only gives an upper bound for the required sample size $N_n$, whose proof (see Section \ref{sec:analysis}) in fact has a plug-in nature. Suppose the sequence of Gaussian-Markov fields $\wn$ in \eqref{eq:KLtrun} converges in some semimetric $d_n$ towards $W^{\M}$ in \eqref{eq:continuumprior} with rate $\mathbb{E}_{\xi} d_n(\wn,W^{\M})=R(N_n)$ for some function $R$. The required sample size is then obtained by matching $R(N_n)$ with $n^{-1}$. In particular, any improvement of the rate $R(N_n)$ shown in Subsection \ref{sec:L infty conv} will lead to a reduction of the required sample size. However, the spectral convergence rate in Proposition \ref{thm:specRate} and hence $R(N_n)$ should not be faster than the resolution of the point cloud, which is shown to scale like $N_n^{-1/m}$ in Proposition \ref{prop:transmap}. Therefore by matching $N_n^{-1/m}$ with $n^{-1}$ we expect a polynomial dependence of $N_n$ on $n$ to be necessary to achieve optimal contraction rate, further demonstrating the need of unlabeled data in our specific setting.

\section{Posterior Contraction: Background and Set-up}\label{sec:background pc}
In this section we present necessary background on posterior contraction rates.  In Subsection \ref{ssec:generalprinciples} we review the main results on posterior contraction that our theory builds on, and in Subsection \ref{ssec:oursetting}  we explain how these general results are used in our proof. 

It will be important to note that our prior is constructed with the $X_i$'s, and we shall view our observations as the $Y_i$'s only, as the notation in \eqref{eq:discrete posterior} suggests. In particular, the $Y_i$'s are independent but non-identically distributed (i.n.i.d.) and hence we will apply general results from \cite{ghosal2007convergence}. This also explains the double expectation in \eqref{eq:pc}, where the randomness of the $X_i$'s is treated separately. 

\subsection{General Principles}\label{ssec:generalprinciples}
Here we review general posterior contraction theory for the problem of estimating $f_0$ at the continuum level from i.n.i.d data, following \cite{ghosal2007convergence}. 
Suppose the data $\{Y_i\}_{i=1}^n$ are generated according to a density $P_{f_0}^{(n)}=\prod_{i=1}^n p_{f_0,i}$ for some ground truth parameter $f_0$, where $p_{f_0,i}$ is the individual density for each observation. Let $\Pi_n$ be a sequence of priors over $f_0$ that is supported on some parameter space $\mathcal{F}$ equipped with a semimetric $d_n$. Theorem 4 of \cite{ghosal2007convergence} states that 
\begin{align}
    \mathbb{E}_{f_0} \Pi_n\left(f\in\mathcal{F}: d_n(f,f_0)\geq M_n\eps_n\,|\, \Y_n\right)\xrightarrow{n\rightarrow \infty} 0, \label{eq:pc general}
\end{align}
provided that there exists sets $\mathcal{F}_n\subset{\mathcal{F}}$ and positive real numbers $\eps_n, K$ so that 
\begin{align}
    \log N(\eps_n, \mathcal{F}_n, d_n) &\leq n\eps_n^2, \label{eq:entropy}\\
    \Pi_n(\mathcal{F}_n^c) &\leq  e^{-(K+4)n\eps_n^2},\label{eq:complement}\\
    \Pi_n(B_n^*(f_0, \eps_n))&\geq e^{-Kn\eps_n^2}, \label{eq:prior mass} 
\end{align}
where $N(\eps_n,\mathcal{F}_n,d_n)$ is the minimum number of $d_n$-balls of radius $\eps_n$ needed to cover $\mathcal{F}_n$. Here
\begin{align}
    B_n^*(f_0,\eps_n):=\left\{f\in\mathcal{F}: \frac{1}{n}\sum_{i=1}^n \dkl(p_{f,i},p_{f_0,i})\leq \eps_n^2, \quad \frac{1}{n}\sum_{i=1}^n v_{\mbox {\tiny{\rm KL}}}(p_{f,i},p_{f_0,i})\leq C\eps_n^2\right\},\label{eq:KL neighbor}
\end{align}
where $C$ is a universal constant, $\dkl(p,q)=\int p \log (p/q)d\gamma$ is the Kullback-Leibler divergence and $v_{\mbox {\tiny{\rm KL}}}(p,q)=\int p|\log(p/q)-\dkl(p,q)|^2d\gamma$.  Conditions \eqref{eq:complement} and \eqref{eq:entropy} roughly state that there are approximating sieves $\mathcal{F}_n$ which capture most of the prior probability while not being too large.
 Condition \eqref{eq:prior mass} requires sufficient prior mass near the truth $f_0$ and together with \eqref{eq:entropy} further indicates that the priors should be ``uniformly spread''. 
  In fact the three conditions are stronger than those in \cite{ghosal2007convergence}[Theorem 4] but will suffice in our case for Gaussian process priors to be discussed shortly below. We shall refer to \cite{ghosal2000convergence}[Section 2] and \cite{ghosal2007convergence}[Section 2] for further discussion on the interpretation and relaxation of these conditions.

The general theory can be used to analyze a wide range of statistical models but does not give a recipe for constructing the sieves $\mathcal{F}_n$ and $\eps_n$. However, when the priors are Gaussian there exists a simple relation between $\eps_n$ and the priors. Suppose in addition that $\Pi_n=\mathcal{L}(w_n)$ are Gaussian measures on some Banach space $(\mathbb{B},\|\cdot\|_{\mathbb{B}})$ that converge to a fixed Gaussian measure $\Pi=\mathcal{L}(w)$ on the same Banach space with $10\mathbb{E}\|w_n-w\|^2_{\mathbb{B}}\leq n^{-1}$. Theorem 2.2 in \cite{van2008rates} then states that the  contraction rate $\varepsilon_n$ can be obtained by studying the \emph{concentration function} of the limit prior $\Pi$, defined as 
\begin{align}\label{eq:concenFunc}
\phi_{f_0}(\eps) := \underset{h\in\mathbb{H}:\|h-f_0\|_{\mathbb{B}}<\eps}{\operatorname{inf}}\, \|h\|^2_{\mathbb{H}} - \log \mathbb{P} (\|w\|_{\mathbb{B}}<\eps),
\end{align}
where $(\mathbb{H},\|\cdot\|_{\mathbb{H}})$ is the  \emph{reproducing kernel Hilbert space} (RKHS) of $\Pi$ (see e.g. \cite{van2008reproducing} for more details on Gaussian measures and the associated RKHSs). More precisely, if $\eps_n$ satisfies $\phi_{f_0}(\eps_n)\leq n\eps_n^2$, then for the same $\eps_n$ and some $\mathcal{F}_n$, the three conditions \eqref{eq:entropy}, \eqref{eq:complement} and \eqref{eq:prior mass} are satisfied (possibly for different proportion constants in front of $n\eps_n^2$) with $d_n$ and $B_n^*(f_0,\eps_n)$ replaced by $\|\cdot\|_{\mathbb{B}}$ and $\{f:\|f-f_0\|_{\mathbb{B}}<\eps_n\},$ respectively.

\subsection{Application to Our Setting} \label{ssec:oursetting}
In this subsection we discuss how we utilize the general theory outlined above, and provide a road map for the proof of Theorem \ref{thm:Main Thm}. Notice that in \eqref{eq:pc} the sequence of posteriors are supported on the discrete space $L^{\infty}(\mu_{N_n})$, whose size changes with $n$. To alleviate this issue we will reduce the analysis to a sequence of posteriors $\Pi_n^{\M}(\cdot\,|\, \X_{N_n},\Y_n)$ supported on the same continuum space $L^{\infty}(\mu)$ in Subsection \ref{sec:reduction}. The corresponding sequence of priors will turn out to have the form  $\Pi_n^{\M}( \cdot\, | \X_{N_n}) = \mathcal{L}(\Phi(\mathcal{I}W_n))$, with $W_n$ defined in \eqref{eq:KLtrun} and $\mathcal{I}$ a suitable interpolation map so that $\mathcal{I}W_n$ is a Gaussian process on $\M$ that approximates $W^{\M}$. Our analysis will then consist of the following steps. We will establish the three conditions \eqref{eq:entropy}, \eqref{eq:complement}, \eqref{eq:prior mass} for $W^{\M}$ in Subsection \ref{sec:continuum}, followed by a convergence rate analysis of $\mathcal{I}W_n$ towards $W^{\M}$ in Subsection \ref{sec:L infty conv}, so that the three conditions are inherited by $\mathcal{I}W_n$. The assumptions on the link function $\Phi$ will then allow us to conclude similar conditions for $\Phi(\mathcal{I}W_n)$. 

The discrepancy measure $d_n$ can be taken as the empirical $L^2$-norm $\|\cdot\|_n$ in the regression case (see e.g. \cite{ghosal2017fundamentals}[Section 8.3.2]) and the root average square Hellinger distance in the classification setting (\cite{ghosal2007convergence}[Section 3]), defined as 
\begin{align}
    d^2_{n,H}(f,f^{\prime})=\frac{1}{n}\sum_{i=1}^n \int \left(\sqrt{p_{f,i}}-\sqrt{p_{f^{\prime},i}}\right)^2 d\gamma. \label{eq:RASH}
\end{align}
In the latter case, one can show that $d_{n,H}$ is equivalent to $\|\cdot\|_n$. Indeed, since the densities are uniformly bounded, $\|f-f^{\prime}\|_n$ is bounded above by a multiple of \eqref{eq:RASH}. Furthermore \eqref{eq:RASH} is upper bounded by a multiple of $\|f-f^{\prime}\|_n$ by \cite{ghosal2017fundamentals}[Lemma 2.8(ii)] given our assumption on $\Phi$. Hence this explains the choice of $\|\cdot\|_n$ in \eqref{eq:pc}.

A natural choice for the Banach space in our setting is then  $\mathbb{B}={L^{\infty}(\mu)}$ since (i) $d_n=\|\cdot\|_{n}\leq \|\cdot\|_{L^{\infty}(\mu)};$ and (ii) $B_n^*(f_0,\eps_n) \supset \{f\in\mathcal{F}: \|f-f_0\|_{L^{\infty}(\mu)} \leq \tilde{C}\eps_n\}$ for some universal constant $\tilde{C}$. The first point is clear and an upper bound on the metric entropy \eqref{eq:entropy} in $\|\cdot\|_{L^{\infty}(\mu)}$ will automatically yield an upper bound in $\|\cdot\|_{n}$.  The second point follows from the fact that one can upper bound the two quantities in \eqref{eq:KL neighbor} by $\|f-f_0\|_n$ and hence $\|f-f_0\|_{L^{\infty}(\mu)}$ in both the regression and classification setting. Therefore the prior mass condition \eqref{eq:prior mass} for $W^{\M}$ in $L^{\infty}(\mu)$ balls is sufficient to give the corresponding condition for $B^{*}(f_0,\eps_n)$. This motivates us to consider the continuum Gaussian field defined in \eqref{eq:continuumprior} as an element in $L^{\infty}(\mu)$ in Subsection \ref{sec:continuum} and the $L^{\infty}(\mu)$ convergence rate in Subsection \ref{sec:L infty conv}.

\section{Proof of the Main Result} \label{sec:analysis}
In this section we prove Theorem \ref{thm:Main Thm}. An important part of the proof is to formalize the convergence of the Gauss-Markov random field $\wn$ in \eqref{eq:KLtrun}, defined in the data cloud $\X_{N_n},$ to the Mat\'ern field $W^{\M}$ in \eqref{eq:continuumprior}, defined in $\M$.  To that end, we introduce an interpolation of $\wn$ into $\M$
\begin{align}
    \Wn:=\sum_{i=1}^{k_{N_n}} \left[1+\lambda_i^{(N_n)}\right]^{-\frac{s}{2}} \xi_i \psi_i^{(N_n)}\circ T_{N_n}, \quad \quad \xi_i\overset{i.i.d.}{\sim} \mathcal{N}(0,1), \label{eq:sequenceprior}
\end{align}
where $T_{N_n}:\M\rightarrow \{X_1,\ldots,X_{N_n}\}$ are transport maps to be specified in Proposition \ref{prop:transmap} below. In Subsection \ref{sec:reduction} we show that it suffices to establish posterior contraction with respect to the prior 
\begin{align}
    \Pim( \cdot \,|\, \X_{N_n}) :=\mathcal{L}(\Phi(\Wn) \,|\, \X_{N_n}). \label{eq:cont graph prior}
\end{align}
 In Subsection \ref{sec:continuum} we show concentration properties of the limiting Gaussian field $W^{\M}.$ In Subsection \ref{sec:L infty conv} we establish the $L^{\infty}(\mu)$-convergence rate of $\Wn$ towards $W^{\M}$.
Combining the facts that $\Wn$ is close to $W^{\M}$ with the contraction properties of $W^{\M}$, we deduce that $\L(\Wn \,|\, \X_{N_n} )$ satisfies the three conditions \eqref{eq:entropy}, \eqref{eq:complement}, \eqref{eq:prior mass}. We complete the proof in Subsection \ref{sec:combine} by  lifting these conditions to $\L(\Phi(\Wn) \,|\, \X_{N_n} )$ in both the regression and classification problems.

\subsection{Reduction via Interpolation}\label{sec:reduction}
Here we show that in order to establish Theorem \ref{thm:Main Thm} it suffices to prove posterior contraction with respect to the continuum prior $\Pim( \cdot \,|\, \X_{N_n})$ defined in \eqref{eq:cont graph prior}. Let $\mathcal{I}:=\mathcal{I}_{N_n}:L^{\infty}(\mu_{N_n})\rightarrow L^{\infty}(\mu)$ be the interpolation map defined by $\mathcal{I}u:=u\circ T_{N_n}$. Specifically, we claim that, in order to establish Theorem \ref{thm:Main Thm}, it suffices to show that
\begin{align*}
    \mathbb{E}_{X}\mathbb{E}_{f_0}\Pim\left(f\in\mathcal{I}(L^{\infty}(\mu_{N_n})):\|f-f_0\|_n\geq M_n\eps_n\,|\, \X_{N_n} , \Y_n\right) \xrightarrow{n\rightarrow\infty} 0.
\end{align*}
The only property of the maps $T_{N_n}$ in \eqref{eq:sequenceprior} that we will use in this subsection is that $T_{N_n}(X_i) = X_i.$  
In order to establish the claim, let $A_{n}=\{f\in L^{\infty}(\mu_{N_n}):\|f-f_0\|_n\geq M_n\eps_{n}\}$. Since $A_n\subset \mathcal{I}^{-1}(\mathcal{I}(A_n))$ we have 
\begin{align*}
    \prior(A_n\,|\,\X_{N_n},\Y_n) \leq \prior(\mathcal{I}^{-1}(\mathcal{I}(A_n))\,|\,\X_{N_n},\Y_n)=\mathcal{I}_{\sharp}[\prior(\cdot\,|\,\X_{N_n},\Y_n)](\mathcal{I}(A_n)),
\end{align*}
where $\mathcal{I}_{\sharp}$ denotes push-forward through the map $\mathcal{I}.$ (Recall that the push-forward of a measure $\nu$ through the map $\mathcal{I}$ is defined by the relationship $(\mathcal{I}_{\sharp}\nu)(B):=\nu(\mathcal{I}^{-1}(B))$.)  Therefore, it follows from Lemma \ref{lemma:commutative} below that 
\begin{equation}
 \prior(A_n\,|\,\X_{N_n},\Y_n)  \le \Pim( \mathcal{I} (A_n) | \X_{N_n}, \Y_n), \label{eq:interpolated posterior}
\end{equation}
and hence it suffices to bound the right-hand side of \eqref{eq:interpolated posterior}. 
Since $T_{N_n}(X_i)=X_i$, we have
\begin{align*}
    \mathcal{I}(A_n)=\{f\circ T_{N_n}: \|f-f_0\|_n\geq M_n\eps_n \}&=\{f\circ T_{N_n}: \|f\circ T_{N_n}-f_0\|_n\geq M_n\eps_n \}\\
    &=\{f\in \mathcal{I}(L^{\infty}(\mu_{N_n})): \|f-f_0\|_n\geq M_n\eps_n\},
\end{align*}
and the claim is established.

\begin{lemma}\label{lemma:commutative}
It holds that  $   \Pim(\cdot\,|\, \X_{N_n},\Y_n)=\mathcal{I}_{\sharp}[ \prior(\cdot\,|\, \X_{N_n}, \Y_n)].$
\end{lemma}
\begin{proof}
{\bf Step 1:}  First we show that $\Pim( \cdot\,| \X_{N_n})  = \mathcal{I}_{\sharp} [ \prior( \cdot\, | \X_{N_n})]. $
Since $\mathcal{I}$ is linear, we have that $\Wn = \mathcal{I}\wn.$ Then observe that $\mathcal{I} \bigl(\Phi(\wn)\bigr)=\Phi \bigl(\mathcal{I}(\wn)\bigr)$. Indeed, for $x\in T_{N_n}^{-1}(\{X_i\})$, we have
\begin{align*}
    \mathcal{I} \bigl(\Phi(\wn)\bigr)(x)=\Phi(\wn)(X_i)=\Phi\bigl(\wn(X_i)\bigr)=\Phi \bigl(\mathcal{I}(\wn)(x)\bigr)=\Phi \bigl(\mathcal{I}(\wn)\bigr)(x).
\end{align*}
Therefore, we have that 
$$\Pim( \cdot\,| \X_{N_n}) = \mathcal{L} \bigl(\Phi(\Wn) | \X_{N_n} \bigr) =\mathcal{L} \bigl(\mathcal{I} \bigl(\Phi(\wn)\bigr) | \X_{N_n} \bigr).$$
Finally, note that $\mathcal{I}_{\sharp} [\prior( \cdot\, | \X_{N_n}) ]=  \mathcal{L} \bigl(\mathcal{I} \bigl(\Phi(\wn)\bigr) | \X_{N_n} \bigr),$ since
\begin{align*}
\mathcal{I}_{\sharp} \prior (B | \X_{N_n}) = \prior \bigl(\mathcal{I}^{-1}(B) | \X_{N_n} \bigr) = \mathbb{P} \bigl(\Phi(\wn)\in \mathcal{I}^{-1}(B) | \X_{N_n} \bigr)= \mathbb{P} \bigl(\mathcal{I}(\Phi(\wn))\in B | \X_{N_n} \bigr).
\end{align*}
{\bf Step 2:} Now we show that $   \Pim(\cdot\,|\, \X_{N_n},\Y_n)=\mathcal{I}_{\sharp}[ \prior(\cdot\,|\, \X_{N_n}, \Y_n)].$
By definition of pushforward measure, it suffices to show that, for any measurable $B,$
\begin{align*}
      \Pim(B\,|\,\X_{N_n},\Y_n)=\prior(\mathcal{I}^{-1}(B)\,|\,\X_{N_n},\Y_n).
\end{align*}
Using Step 1, the left-hand side equals 
\begin{align}
    \Pim(B\,|\,\X_{N_n},\Y_n) 
    &=
    \frac{\int_B \prod_{i=1}^n L_{Y_i|X_i}(f) \,d\mathcal{I}_{\sharp}[\prior(f \,|\,\X_{N_n})]}{\int_{\mathcal{I}(L^{\infty}(\mu_{N_n}))} \prod_{i=1}^n L_{Y_i|X_i}(f)\, d\mathcal{I}_{\sharp}[\prior(f\,|\,\X_{N_n})]}, \label{eq:Aeq1}
\end{align}
where 
$L_{Y_i|X_i}(f)$ is given in \eqref{eq:discrete likelihood 1} and \eqref{eq:discrete likelihood 2}. Note that pointwise values of $f$ are well-defined since $\mathcal{I}_{\sharp}\prior$ is supported on $\mathcal{I}(L^{\infty}(\mu_{N_n}))$.
By the change-of-variable formula for pushforward measures, 
\begin{align*}
    \eqref{eq:Aeq1} = \frac{\int_{\mathcal{I}^{-1}(B)} \prod_{i=1}^n L_{Y_i|X_i}\circ \mathcal{I}(f_n) \,d\prior(f_n\,|\,\X_{N_n}) }{\int_{L^{\infty}(\mu_{N_n})} \prod_{i=1}^n L_{Y_i|X_i}\circ \mathcal{I}(f_n)\,d\prior(f_n\,|\,\X_{N_n})},
\end{align*}
which equals \eqref{eq:discrete posterior} with $B$ replaced by $\mathcal{I}^{-1}(B)$, by noticing that $L_{Y_i|X_i}\circ\mathcal{I}(f_n)$ is exactly the same as in \eqref{eq:discrete likelihood 1} and \eqref{eq:discrete likelihood 2}. The result follows. 
\end{proof}



\subsection{Regularity and Contraction Properties of the Limiting Field}\label{sec:continuum}
In this subection we study  the limit Gaussian measure $\pi=\mathcal{N}(0,(I-\Delta)^{-s})=\mathcal{L}(W^{\M})$. Recall that the samples admit Karhunen-Lo\`eve expansion
\begin{align}
    W^{\M}=\sum_{i=1}^{\infty}(1+\lambda_i)^{-\frac{s}{2}}\xi_i \psi_i,\quad \quad \xi_i\overset{i.i.d.}{\sim} \mathcal{N}(0,1), \label{eq:continuumprior2}
\end{align}
where $(\lambda_i,\psi_i)$'s are the eigenpairs of the Laplace-Beltrami operator $-\Delta$ on $\M$. Notice that a larger $s$ leads to a faster decay of the coefficients and hence more regular sample paths. From  Weyl's law that $\lambda_i\asymp i^{\frac{2}{m}}$ (see e.g. \cite{canzani2013analysis}[Theorem 72]), setting $s>\frac{m}{2}$ makes $\pi$ a well-defined measure on $L^2(\mu)$. To ensure almost sure continuity of the samples, so that point evaluations are well-defined and $\pi$ is supported over $L^{\infty}(\mu),$ we need the stronger assumption that $s>m-\frac{1}{2}$. 
\begin{lemma}\label{lemma:almost sure cty}
If $s>m-\frac12$, then samples of $\pi$ are almost surely continuous. 
\end{lemma} 
\begin{proof}
By \cite{lang2016continuity}[Corollary 4.5], it suffices to show that
 \begin{align*}
    \mathbb{E}_{\xi} |W^{\M}(x)-W^{\M}(y)|^2 \leq Cd_{\M}(x,y)^{\eta},
\end{align*}
for some $\eta>0$. We have
\begin{align*}
    \mathbb{E}_{\xi}|W^{\M}(x)-W^{\M}(y)|^2 & =\mathbb{E}_{\xi}\left(\sum_{i=1}^{\infty}(1+\lambda_i)^{-\frac{s}{2}}\xi_i(\psi_i(x)-\psi_i(y))\right)^2 \\
    &=\sum_{i=1}^{\infty}(1+\lambda_i)^{-s}|\psi_i(x)-\psi_i(y)|^2,
\end{align*}
since $\mathbb{E}_{\xi}\xi_i\xi_j=0$ for $i\neq j$. By Proposition \ref{thm:efunc gradient bound},
\begin{align*}
    |\psi_i(x)-\psi_i(y)| &\leq C\lambda_i^{\frac{m-1}{4}},\\
    |\psi_i(x)-\psi_i(y)| &\leq \|\nabla \psi_i \|_{L^{\infty}(\mu)} d_{\M}(x,y) \leq C\lambda_i^{\frac{m+1}{2}} d_{\M}(x,y). 
\end{align*}
Using Weyl's law, we have  
\begin{align*}
    \mathbb{E}_{\xi}|W^{\M}(x)-W^{\M}(y)|^2 &\lesssim \sum_{i=1}^{\infty} (1+\lambda_i)^{-s} \operatorname{min}\left\{\lambda_i^{m+1}d_{\M}(x,y)^2,\lambda_i^{\frac{m-1}{2}}\right\} \\
    & \lesssim \sum_{i=1}^{\infty}i^{-\frac{2s}{m}} \operatorname{min}\left\{i^{\frac{2m+2}{m}}d_{\M}(x,y)^2, i^{\frac{m-1}{m}}\right\}\\
    & \lesssim d_{\M}(x,y)^2\int_{1}^{K} z^{\frac{-2s+2m+2}{m}} dz +\int_{K}^{\infty} z^{\frac{-2s+m-1}{m}}dz.
\end{align*}
Setting $K= d_{\M}(x,y)^{-\frac{2m}{m+3}}$, we have 
\begin{align*}
    \mathbb{E}_{\xi}|W^{\M}(x)-W^{\M}(y)|^2 &\lesssim  d_{\M}(x,y)^{\frac{4s-4m+2}{m+3}}.
\end{align*}
The result  follows since  $s>m-\frac12$.
\end{proof}

\begin{remark}
The argument above relies on the worst case $L^{\infty}(\mu)$ bound on the eigenfunctions of $-\Delta$ given in Proposition \ref{thm:efunc gradient bound}. If the eigenfunctions are uniformly bounded (which holds for the family of flat manifolds \citep{toth2002riemannian} that includes, for instance, the torus), then continuity can be guaranteed for $s>\frac{m}{2}$. 
\end{remark}

From now on, we shall consider $\pi$ as a measure over $L^{\infty}(\mu)$ with continuous sample paths, where pointwise evaluation is well-defined. Using the series representation \eqref{eq:continuumprior2}, the reproducing kernel Hilbert space $\mathbb{H}$ associated with $\pi$ has the following characterization 
\begin{align}
	\mathbb{H}=\left\{h=\sum_{i=1}^{\infty} \ h_i  \psi_i: \|h\|_{\mathbb{H}}^2:=\sum_{i=1}^{\infty}  h_i^2   (1+\lambda_i)^s <\infty \right\}. \label{eq:rkhs}
\end{align}
From the general theory in \cite{van2008rates}, the concentration properties of $W$ can be characterized by the concentration function 
\begin{align}\label{eq:concenFunc}
\phi_{w_0}(\eps) := \underset{h\in\mathbb{H}:\|h-w_0\|_{L^{\infty}(\mu)}<\eps}{\operatorname{inf}}\, \|h\|^2_{\mathbb{H}} - \log \mathbb{P} (\|W\|_{L^{\infty}(\mu)}<\eps),
\end{align}
where $w_0$ belongs to the support of $\pi$.  For our purpose, $w_0$ will be set as $\Phi^{-1}(f_0)$. The three conditions \eqref{eq:entropy}, \eqref{eq:complement} and \eqref{eq:prior mass} hold for the $\eps_n$ that satisfies $\phi_{w_0}(\eps_n)\leq n\eps_n^2$. 

Before stating our main result in this subsection, we follow \cite{castillo2014thomas,coulhon2012heat} to define a Besov space which will be needed to characterize the regularity of $w_0$. Let $\Psi$ be an even function in the Schwartz space $\mathcal{S}(\mathbb{R})$ with 
\begin{align*}
    0\leq \Psi\leq 1,\quad \Psi \equiv 1\,\, \text{on}\,\, [-\frac{1}{2},\frac{1}{2}], \quad \text{supp}(\Psi)\subset[-1,1].
\end{align*}
Define the Besov space 
\begin{align*}
    B_{\infty,\infty}^{\beta} :=\left\{w: \|w\|_{B_{\infty,\infty}}^{\beta}:= \underset{j\in\mathbb{N}}{\operatorname{sup}\,} 2^{\beta j}\|\Psi_j(\sqrt{\Delta})w(\cdot)-w(\cdot)\|_{L^{\infty}(\mu)} <\infty\right\},
\end{align*}
where $\Psi_j(\cdot)=\Psi(2^{-j}\cdot)$ and, for $w=\sum_{i=1}^{\infty}w_i \psi_i,$ 
\begin{align*}
    \Psi_j(\sqrt{\Delta})w:=\sum_{i=1}^{\infty} \Psi_j(\sqrt{\lambda_i})w_i\psi_i.
\end{align*}
It was shown in \cite{coulhon2012heat}[Proposition 6.2] that the definition is independent of the choice of $\Psi$.

\begin{theorem}\label{thm:cont rate}
Suppose $w_0\in B_{\infty,\infty}^{\beta}$. Consider the prior $\pi=\mathcal{N}(0,(I-\Delta)^{-s})$ with  $s>\beta\wedge m-\frac12$.  Then there exists sets $B_n\subset L^{\infty}(\mu)$ so that
\begin{align*}
    \log N(3\eps_n,B_n,\|\cdot\|_{L^{\infty}(\mu)})&\leq  6Cn\eps_n^2, \\
    \pi(B_n^c)&\leq e^{-Cn\eps_n^2}, \\
    \pi(\|w-w_0\|_{L^{\infty}(\mu)}<2\eps_n)&\geq e^{-n\eps_n^2},
\end{align*}
where $\eps_n$ is a multiple of $n^{-\frac{(s-\frac{m}{2})\wedge \beta}{2s}}(\log n)^{\frac{(s-\frac{m}{2})\wedge \beta}{4s-2m}}$  and $C>1$ is a constant satisfying $e^{-Cn\eps_n^2}<\frac12$. 
\end{theorem}
\begin{proof}
As noted above, it suffices to show that the above $\eps_n$ satisfies $\phi_{w_0}(\eps_n) \leq n\eps_n^2,$ where $\phi_{w_0}$ is defined in \eqref{eq:concenFunc}. To bound the small ball probability, it suffices by general results from \cite{li1999approximation} to bound the metric entropy of $\mathbb{H}_1$, the unit ball in $\mathbb{H}$. Notice that \eqref{eq:rkhs} has the form of a Sobolev ball of regularity $s$. Classical results for Sobolev spaces on $\mathbb{R}^m$ such as \cite{edmunds1996function}[Theorem 3.3.2] have shown that the entropy in the $L^{\infty}$ metric is bounded by a multiple of $\eps^{-m/s}$. The manifold case has been treated in \cite{kushpel2015entropy}[Theorem 3.12]  assuming homogeneity, which includes an additional logarithmic factor, i.e., 
\begin{align}
    \log N(\eps, \mathbb{H}_1,\|\cdot\|_{L^{\infty}(\mu)})\lesssim \eps^{-\frac{m}{s}}\left(\log \frac{1}{\eps}\right)^{\frac12}. \label{eq:metric entropy}
\end{align}
The original theorem was about entropy numbers but can be translated to the above statement on metric entropy and its proof suggested that when $p=2$  the theorem holds with $s>\frac{m}{2}$. In particular their Sobolev class is defined as
\begin{align*}
    I_s U_2:=\left\{h=h_1+\sum_{i=2}^{\infty}\lambda_i^{-\frac{s}{2}}h_i\psi_i: \sum_{i=1}^{\infty}h_i^2\leq 1\right\},
\end{align*}
which is compatible with our $\mathbb{H}$. (We remark that although \cite{kushpel2015entropy} defined their Sobolev class with a further mean-zero condition, their proof actually applied to the general case.) 
Now by \eqref{eq:metric entropy} and \cite{li1999approximation}[Theorem 1.2] we have
\begin{align}
    -\log \mathbb{P}(\|w\|_{L^{\infty}(\mu)}<\eps) \lesssim \eps^{-\frac{2m}{2s-m}} \left(\log \frac{1}{\eps}\right)^{\frac{s}{2s-m}}. \label{eq:small ball}
\end{align}

For the decentering function, let $C_0=\|w_0\|_{B^{\beta}_{\infty,\infty}}$ and consider $h=\Psi_J(\sqrt{\Delta})w_0$ with $J$ large enough so that $C_02^{-\beta J}\leq \eps$. Since $w_0\in B_{\infty,\infty}^{\beta}$, we have
\begin{align}
    \|\Psi_j(\sqrt{\Delta})w_0-w_0\|_{L^{\infty}(\mu)}\leq C_02^{-\beta j}, \quad \quad j\in\mathbb{N}. \label{eq:besov error decomp} 
\end{align}
In particular $\|h-w_0\|_{L^{\infty}(\mu)}\leq C_0 2^{-\beta J}\leq \eps$. Suppose $w_0$ has the representation $w_0=\sum_{i=1}^{\infty}w_i\psi_i$. We then have $h=\sum \Psi_J(\sqrt{\lambda_i}) w_i \psi_i\in \mathbb{H}$ since $\Psi_J(\sqrt{\lambda_i})=0$ for $\sqrt{\lambda_i}>2^J$. Moreover, since $\Psi_J\leq 1$,
\begin{align}
    \|h\|_{\mathbb{H}}^2 \leq \sum_{\sqrt{\lambda_i}\leq 2^J} w_i^2 (1+\lambda_i)^s\leq\sum_{\sqrt{\lambda_i}\leq 1}2^sw_i^2+\sum_{j=1}^J \sum_{2^{j-1}<\sqrt{\lambda_i}\leq 2^{j}} w_i^2(1+\lambda_i)^s. \label{eq:h hnorm}
\end{align}
By \eqref{eq:besov error decomp}, we have 
\begin{align*}
    \sum_{2^{j-1}< \sqrt{\lambda_i}\leq 2^j} w_i^2(1+\lambda_i)^s \lesssim 2^{2js}\sum_{2^{j-1}< \sqrt{\lambda_i}} w_i^2 &\leq2^{2js} \|\Psi_{j-1}(\sqrt{\Delta})w_0-w_0\|_2^2\\
    &\lesssim 2^{2js} \|\Psi_{j-1}(\sqrt{\Delta})w_0-w_0\|_{\infty}^2\lesssim 2^{2(s-\beta)j}.
\end{align*}
Therefore recalling that $C_02^{-\beta J}\leq \eps$,
\begin{align*}
    \eqref{eq:h hnorm}\lesssim \sum_{\sqrt{\lambda_i}\leq 1}2^sw_i^2 +\sum_{j=1}^J 2^{2(s-\beta)j}\lesssim 2^{2(s-\beta)J}\lesssim \eps^{-\frac{2(s-\beta)}{\beta}},
\end{align*}
since the first sum remains bounded as $J\rightarrow \infty$. Combining the above with \eqref{eq:small ball}, we get 
\begin{align*}
    \phi_{w_0}(\eps) \lesssim  \eps^{-\frac{2m}{2s-m}} \left(\log \frac{1}{\eps}\right)^{\frac{s}{2s-m}} + \eps^{-\frac{2(s-\beta)}{\beta}},
\end{align*}
which implies 
\begin{align*}
    \frac{\phi_{w_0}(\eps)}{\eps^2}\lesssim \eps^{-\frac{2s}{s-\frac{m}{2}}} \left(\log \frac{1}{\eps}\right)^{\frac{s}{2s-m}} + \eps^{-\frac{2s}{\beta}}.
\end{align*}
Therefore by choosing 
\begin{align*}
    \eps_n =Cn^{-\frac{(s-\frac{m}{2}\wedge \beta)}{2s}} (\log n)^{\frac{(s-\frac{m}{2})\wedge \beta}{4s-2m}}    
\end{align*}
for a large enough constant $C$, the condition $\phi_{w_0}(\eps_n)\leq n\eps_n^2$ is satisfied. The three assertions then follow from \cite{van2008rates}[Theorem 2.1].
\end{proof}

\subsection{Convergence of Gaussian Fields in $L^{\infty}(\mu)$} \label{sec:L infty conv}
In this subsection we establish $L^\infty(\mu)$ convergence  of $\Wn$ to $W^{\M}.$  We shall denote $N:=N_n$ throughout the rest of this subsection to simplify our notation.
Recall that 
\begin{align}
\begin{split}
    \Wn&=\sum_{i=1}^{k_{N}} \left[1+\lambda_i^{(N)}\right]^{-\frac{s}{2}} \xi_i \psi_i^{(N)}\circ T_{N}, \quad \quad \xi_i\overset{i.i.d.}{\sim} \mathcal{N}(0,1),  \\
     W^{\M} &=\sum_{i=1}^{\infty}(1+\lambda_i)^{-\frac{s}{2}}\xi_i \psi_i,\quad \quad \xi_i\overset{i.i.d.}{\sim} \mathcal{N}(0,1),
     \end{split}
\end{align}
In order to show the convergence, the transportation maps $T_{N}$ will be assumed to be close to the identity in the sense made precise in the following proposition from \cite{trillos2019error}[Theorem 2].
\begin{proposition} \label{prop:transmap}
For $\gamma>1$, there exists a transportation map $T_{N}:\M\rightarrow \{X_1,\ldots,X_{N}\}$  satisfying $T_{N_n}(X_i)=X_i$  so that, with probability $1-O(N^{-\gamma})$,
\begin{align}
   \rho_{N}:=\underset{x\in \mathcal{M}}{\operatorname{sup}}\, d_\M (x,T_{N}(x))\lesssim \frac{(\log N)^{p_m}}{N^{1/m}}, \label{eq:dinfty}
\end{align}
where $p_m = \frac34$ if $m=2$ and $p_m=\frac{1}{m}$ otherwise. 
\end{proposition}
The transport map $T_{N_n}$ is a measure-preserving transformation in the sense that $\mu(T_{N_n}^{-1}(U))=\mu_{N_n}(U)$ for all $U\subset \M$ measurable.  The additional requirement $T_{N}(X_i)=X_i$ that was absent in \cite{trillos2019error}[Theorem 2] is valid here  since modifying $T_{N}$ on a set of $\mu$-measure 0 does not affect the measure preserving property and \eqref{eq:dinfty} still holds in this case. The scaling in \eqref{eq:dinfty} can be thought as the resolution of the point cloud,   and is important in suitably defining the choice of connectivity $\zeta_{N}$, as we shall see. 


The main result of this subsection is the following. 
\begin{theorem}\label{thm:L infty conv}
Let $\delta>0$ be arbitrary.
\begin{enumerate}
    \item $ \frac{3}{2}m-\frac12<s\leq \frac{9}{2}m+\frac52$. Set
    \begin{align*}
    \begin{cases}
    \zeta_N\asymp N^{-\frac{1}{m+4+\delta}}(\log N)^{\frac{p_m}{2}}, \quad  k_N\asymp N^{\frac{1}{(m+4+\delta)(6+\frac{6}{m})}}(\log N)^{-\frac{mp_m}{12m+12}}, \quad \quad & m\leq 4,\\
    \zeta_N\asymp N^{-\frac{1}{2m}}(\log N)^{\frac{p_m}{2}},\quad  k_N\asymp N^{\frac{1}{12m+12}}(\log N)^{-\frac{mp_m}{12m+12}}, \quad \quad &m\geq 5.
    \end{cases}
\end{align*}
Then, with probability tending to 1,
\begin{align*}
    \mathbb{E}_{\xi}\|\Wn-W^{\M} \|^2_{L^{\infty}(\mu)} \lesssim
    \begin{cases}
    N^{-\frac{2s-3m+1}{(m+4+\delta)(6m+6)}}(\log N)^{\frac{p_m(2s-3m+1)}{12m+12}}, \quad \quad &m\leq 4,\\
    N^{-\frac{2s-3m+1}{m(12m+12)}} (\log N)^{\frac{p_m(2s-3m+1)}{12m+12}},\quad \quad &m\geq 5.
    \end{cases}
\end{align*}

\item $s>\frac{9}{2}m+\frac52$. Set
\begin{align*}
    \begin{cases}
    \zeta_N\asymp N^{-\frac{1}{m+4+\delta}}(\log N)^{\frac{p_m}{2}},\quad  k_N\asymp N^{\frac{m}{(m+4+\delta)(2s-3m+1)}}(\log N)^{-\frac{mp_m}{4s-6m+2}}, \quad \quad  & m\leq 4, \\
    \zeta_N\asymp N^{-\frac{1}{2m}}(\log N)^{\frac{p_m}{2}},\quad  k_N\asymp N^{\frac{1}{4s-6m+2}}(\log N)^{-\frac{mp_m}{4s-6m+2}}, \quad \quad & m\geq 5.
    \end{cases}
\end{align*}
Then, with probability tending to 1,
\begin{align*}
    \mathbb{E}_{\xi}\|\Wn-W^{\M} \|^2_{L^{\infty}(\mu)} \lesssim
    \begin{cases}
    N^{-\frac{1}{m+4+\delta}}(\log N)^{\frac{p_m}{2}},\quad \quad &m\leq 4,\\
    N^{-\frac{1}{2m}}(\log N)^{\frac{p_m}{2}},\quad \quad &m\geq 5.
    \end{cases}
\end{align*}
\end{enumerate}
\end{theorem}

We remark that the statement ``with probability tending to 1'' refers to the randomness coming from the $X_i$'s. 
By solving for $N$ so that the rate matches $n^{-1}$, we get the following. 
\begin{corollary}\label{cor:L infty conv}
Let $\delta>0$ be arbitrary. Let
\begin{align*}
    N\asymp n^{\alpha_1}(\log n)^{\alpha_2},
\end{align*}
where 
\begin{align*}
    \begin{cases}
    \alpha_1=\frac{(m+4+\delta)(6m+6)}{2s-3m+1},\quad  \alpha_2=\frac{p_m(m+4+\delta)}{2},\quad \quad  &\text{if}\quad m\leq 4 \,\,,\,\,\frac32 m-\frac12<s\leq \frac{9}{2}m+\frac52,\\
    \alpha_1=\frac{m(12m+12)}{2s-3m+1},\quad  \alpha_2=mp_m,\quad \quad  &\text{if}\quad m \geq 5 \,\, ,\,\, \frac32 m-\frac12<s\leq \frac{9}{2}m+\frac52,\\
    \alpha_1=m+4+\delta,\quad  \alpha_2=\frac{p_m(m+4+\delta)}{2}, \quad \quad  &\text{if}\quad m\leq 4 \,\, ,\,\, s> \frac{9}{2}m+\frac52,\\
    \alpha_1=2m,\quad  \alpha_2=mp_m, \quad \quad  &\text{if}\quad m\geq 5 \,\, ,\,\, s> \frac{9}{2}m+\frac52.
    \end{cases}
\end{align*}
Let $\zeta_N$ and $k_N$ have the same scaling as in Theorem \ref{thm:L infty conv}. Then, with probability tending to 1,
\begin{align*}
    \mathbb{E}_{\xi}\|\Wn-W^{\M} \|^2_{L^{\infty}(\mu)} \lesssim n^{-1}.
\end{align*}
\end{corollary}

\begin{remark}
For flat manifolds the results of Theorem \ref{thm:L infty conv} and Corollary \ref{cor:L infty conv} can be shown to hold for $s>m$ with corresponding modifications in the scaling of the parameters.
\end{remark}

The key to show the above results is to derive convergence rates for 
\begin{align*}
   |\lambda^{(N)}_i-\lambda_i|,\quad \quad \text{and} \quad \quad \|\psi_i^{(N)}\circ T_N-\psi_i\|_{L^{\infty}(\mu)}.
\end{align*}
We shall build our analysis on the existing results from the literature.
Recall that $\zeta_N$ is the connectivity of the graph and the resolution $\rho_N$  defined in \eqref{eq:dinfty}. Assuming we are in the event that \eqref{eq:dinfty} holds, the following results from \cite{sanz2020spde}[Theorem 4.6 \& 4.7] bound the spectral approximations. 
\begin{proposition}[Spectral Approximation] \label{thm:specRate}
Suppose  $\rho_N\ll \zeta_N$ and $\zeta_{N}\sqrt{\lambda_{k_N}}\ll 1$ for $N$ large. Then there exists orthonormalized eigenfunctions $\{\psi_i^{(N)}\}_{i=1}^{N}$ and $\{\psi_i\}_{i=1}^{\infty}$ so that, for $i=1,\ldots,k_N,$
\begin{align}
    |\lambda_{i}^{(N)}-\lambda_{i}| &\lesssim \lambda_{i}\left(\frac{\rho_N}{\zeta_{N}}+\zeta_{N}\sqrt{\lambda_{i}}\right),\label{eq:eval rate}\\
    \|\psi_{i}^{(N)}\circ T_N-\psi_{i}\|_{L^2(\mu)}^2 &\lesssim  i^3  \left(\frac{\rho_N}{\zeta_{N}}+\zeta_{N}\sqrt{\lambda_{i}}\right). \label{eq:efun rate}
\end{align}
\end{proposition}

\begin{remark}
Proposition \ref{thm:specRate} bounds the eigenfunction approximation in $L^2(\mu)$ norm and we need to lift such result to the $L^{\infty}(\mu_N)$ norm. We remark that $L^{\infty}(\mu_N)$ convergence rates have already been established in the literature (see e.g. \cite{dunson2019spectral,calder2020lipschitz}). However these results do not contain an explicit proportion constant in terms of the index $i$ as in \eqref{eq:efun rate}. Instead of going through their details to find the explicit constants, which would make our presentation much more involved, we directly show $L^{\infty}(\mu_N)$ bounds based on Proposition \ref{thm:specRate}, by following the same idea as in \cite{calder2020lipschitz}. Since we build our results from \eqref{eq:efun rate}, which is not the sharpest bound in the literature, our results in Theorem \ref{thm:L infty conv} and Corollary \ref{cor:L infty conv} suffer the same drawback. Nevertheless, the goal of this paper is to demonstrate the idea that unlabeled data helps and hence the finding sharpest scaling of $N_n$ is less essential.  
\end{remark}

Below we record four results from \cite{calder2019improved}[Theorem 3.3], \cite{donnelly2006eigenfunctions}[Theorem 1.2] and \cite{xu2005asymptotic}[Equation (2.10)],  \cite{calder2020rates}[Corollary 2.5] that will be needed. 

\begin{proposition}[Pointwise Error of $\Delta_N$] \label{thm:pointwise}
Let $f\in \mathcal{C}^3(\M)$. Then with probability  $1-2n\exp(-cN\zeta_N^{m+4})$,
\begin{align*}
    \|\Delta_Nf(x)-\Delta f(x)\|_{L^{\infty}(\mu_N)}\leq C(1+\|f\|_{\mathcal{C}^3(\M)})\zeta_N.
\end{align*}
\end{proposition}

\begin{proposition}[Bounds on Eigenfunctions and Their Gradients] \label{thm:efunc gradient bound}
Let $\psi$ be an $L^2(\mu)$-normalized eigenfunction of $-\Delta$ associated with eigenvalue $\lambda\neq 0$. Then, for $k\in\mathbb{N},$
\begin{align*}
     \|\psi\|_{L^{\infty}(\mu)} &\leq C \lambda^{\frac{m-1}{4}}, \\
    \|\nabla^k \psi\|_{L^{\infty}(\mu)}&\leq C\lambda^{k+\frac{m-1}{2}}.
\end{align*}
\end{proposition}

\begin{proposition}[$L^{\infty}$ Bound In Terms of $L^1$ Bound] \label{thm:Lip graph efun}
Suppose $\zeta_N\leq \frac{c}{\Lambda+1}$ where $c$ is a sufficiently small constant depending only on $\M$. Then with probability $1-C\zeta_N^{-6m}\exp(-cN\zeta_N^{m+4})-2N\exp(-cN(\Lambda+1)^{-m})$ 
\begin{align*}
    \|u\|_{L^{\infty}(\mu_N)}\leq C(\Lambda+1)^{m+1}\|u\|_{L^1(\mu_N)}
\end{align*}
for all $u\in L^2(\mu_N)$ with $\lambda_u<\Lambda,$ where 
\begin{align}
    \lambda_u :=\frac{\|\Delta_N u\|_{L^{\infty}(\mu_N)}}{\|u\|_{L^{\infty}(\mu_N)}}.\label{eq:lambda_u}
\end{align}
\end{proposition}
Notice that the high probability condition is satisfied only if 
\begin{align*}
    \zeta_N^{-6m}\exp(-cN\zeta_N^{m+4})-2N\exp(-cN(\Lambda+1)^{-m})\rightarrow 0
\end{align*}
and this requires some care in setting the scaling of $\zeta_N$ and $k_N$. 
In particular a sufficient condition for $\zeta_N$ is $\zeta_N \gtrsim N^{-\frac{1}{m+4+\delta}}$ for an arbitrarily small $\delta>0$.  On the other hand the condition for $\Lambda$ reduces to the scaling of $k_N$. 
Indeed if $u=\psi_i^{(N)}$, then $\lambda_u=\lambda^{(N)}_i$ and therefore $\Lambda$ can be chosen as a multiple of $\lambda_i$ given the spectral approximation in Proposition \ref{thm:specRate}. We will show that the same choice of $\Lambda$ suffices when $u=\psi_i^{(N)}-\psi_i$. Since we are interested in bounding the first $k_N$ eigenfunctions, we can set $\Lambda$ to be multiple of $\lambda_{K_N}$. Therefore in order to make $2N\exp\left(-cN(\Lambda+1\right)^{-m})\ll 1$, it is sufficient to have $\lambda_{k_N} \lesssim N^{\frac{1-\delta}{m}}$, i.e., if $k_N \lesssim N^{\frac{1-\delta}{2}}$ by Weyl's law. Therefore we should keep the following in mind when we set the scaling later
\begin{align}
     N^{-\frac{1}{m+4+\delta}}\lesssim \zeta_N,\quad \quad k_N \lesssim N^{\frac{1-\delta}{m}},\quad \quad \zeta_N k_N^{\frac{2}{m}}\lesssim 1, \label{eq:para scale}
\end{align}
where the third requirement corresponds to the condition that $\zeta_N (\Lambda+1)\leq c$.

Now we are ready to show the $L^{\infty}(\mu_N)$ bound of eigenfunction approximations.  
\begin{lemma}\label{lemma:L infty rate}
Under the same conditions and the intersection of the high probability events as in Propositions \ref{thm:specRate} and \ref{thm:Lip graph efun} we have, for $i=1,\ldots,k_N,$
\begin{align*}
    \|\psi_i^{(N)}\circ T_N-\psi_i\|_{L^{\infty}(\mu)} \lesssim  \lambda_i^{m+1}i^\frac{3}{2}\sqrt{\frac{\rho_N}{\zeta_N}+\zeta_N\sqrt{\lambda_i}}.
\end{align*}

\end{lemma}
\begin{proof}
We will first modify the $L^2(\mu)$-bound \eqref{eq:efun rate} to an $L^2(\mu_N)$-bound using the regularity of eigenfunctions of $-\Delta$, which is then lifted to an $L^{\infty}(\mu_N)$-bound using Proposition \ref{thm:pointwise} and \ref{thm:Lip graph efun}. Finally we use regularity of the eigenfuctions again to transfer the $L^{\infty}(\mu_N)$-bound to an $L^{\infty}(\mu)$-bound. To start, we notice that the transport $T_N$ induces a partition of $\M$ by the sets $\{U_i=T_N^{-1}(\{X_i\})\}_{i=1}^N$. Furthermore, the measure preserving property gives $\mu(U_i)=\frac{1}{N}$ and Proposition \ref{prop:transmap} implies that $U_i\subset B_{\M}(X_i,\rho_N)$. We then have
\begin{align*}
    \|\psi_i^{(N)}-\psi_i\|^2_{L^2(\mu_N)}=\frac{1}{N}\sum_{i=1}^N  & |\psi_i^{(N)}(X_i)-\psi_i(X_i)|^2=\sum_{i=1}^N \int_{U_i}|\psi_i^{(N)}(X_i)-\psi_i(X_i)|^2 d\mu(x)  \\
    & \leq \sum_{i=1}^N \int_{U_i} 2|\psi_i^{(N)} (X_i) -\psi_i(x)|^2+2|\psi_i(x)-\psi_i(X_i)|^2 d\mu(x),
\end{align*}
where 
\begin{align*}
    \int_{U_i} |\psi_i(x)-\psi_i(X_i)|^2 d\mu(x) \leq \frac{1}{N} \|\nabla \psi_i\|^2_{\infty}\,\rho_N^2.
\end{align*}
Hence by Proposition \ref{thm:efunc gradient bound} and Weyl's law
\begin{align*}
    \|\psi_i^{(N)}-\psi_i\|^2_{L^2(\mu_N)}\leq 2\|\psi_i^{(N)}\circ T_N-\psi_i\|^2_{L^2(\mu)} +Ci^{\frac{m+1}{m}}\rho_N^2.
\end{align*}
Since $i^{\frac{m}{m+1}}\rho_N^2$ is of higher order than \eqref{eq:efun rate}, we conclude that 
\begin{align}
    \|\psi_i^{(N)}-\psi_i\|^2_{L^2(\mu_N)} \lesssim  i^3  \left(\frac{\rho_N}{\zeta_{N}}+\zeta_{N}\sqrt{\lambda_{i}}\right). \label{eq: L2 mu_n rate}
\end{align}
Now let $g=\psi_i^{(N)}-\tilde{\psi}_i$ with $\tilde{\psi_i}=\psi_i|_{\{X_1,\ldots,X_N\}}$. We have by Proposition \ref{thm:pointwise}
\begin{align*}
    \|\Delta_N g\|_{L^{\infty}(\mu_N)} &\leq \|\Delta_N \psi_i^{(N)} - \Delta \psi_i\|_{L^{\infty}(\mu_N)} +\| \Delta \psi_i - \Delta_N \tilde{\psi}_i\|_{L^{\infty}(\mu_N)}\\
    &\leq \|\lambda_i^{(N)}\psi_i^{(N)}-\lambda_i \psi_i\|_{L^{\infty}(\mu_N)}+C(1+\|\psi_i\|_{\mathcal{C}^3(\M)})\zeta_N\\
    &\leq \lambda_i^{(N)}\|\psi_i^{(N)}-\psi_i\|_{L^{\infty}(\mu_N)} + |\lambda_i^{(N)}-\lambda_i|\|\psi_i\|_{L^{\infty}(\mu_N)}+ C(1+\|\psi_i\|_{\mathcal{C}^3(\M)})\zeta_N\\
    &\leq \lambda_i^{(N)}\|g\|_{L^{\infty}(\mu_N)}+C\lambda_i^{\frac{m+5}{2}}\left(\frac{\rho_N}{\zeta_N}+\zeta_N\right),
\end{align*}
where we have used \eqref{eq:eval rate} and Proposition \ref{thm:efunc gradient bound} in the last step. Therefore recalling the definition in \eqref{eq:lambda_u}, we have 
\begin{align*}
    \lambda_g \leq \lambda_i^{(N)} +C\|g\|^{-1}_{L^{\infty}(\mu_N)}\lambda_i^{\frac{m+5}{2}}\left(\frac{\rho_N}{\zeta_N}+\zeta_N\right). 
\end{align*}
If $\lambda_g\geq \lambda_i^{(N)}+1$, then we have 
\begin{align}
    \|g\|_{L^{\infty}(\mu_N)}\leq C\lambda_i^{\frac{m+5}{2}}\left(\frac{\rho_N}{\zeta_N}+\zeta_N\right). \label{eq:g bound 1} 
\end{align}
Otherwise if $\lambda_g\leq \lambda_i^{(N)}+1$, then Proposition \ref{thm:Lip graph efun} and \eqref{eq: L2 mu_n rate} implies 
\begin{align}
    \|g\|_{L^{\infty}(\mu_N)}\leq C\left[\lambda_i^{(N)}+1\right]^{m+1}\|g\|_{L^1(\mu_N)}\leq C\lambda_i^{m+1}\|g\|_{L^2(\mu_N)} \leq C\lambda_i^{m+1}i^\frac{3}{2}\sqrt{\frac{\rho_N}{\zeta_N}+\zeta_N\sqrt{\lambda_i}},  \label{eq:g bound 2}
\end{align}
where we have used the fact that $\lambda_i^{(N)}\leq C\lambda_i$ and $\|g\|_{L^1(\mu_N)}\leq \|g\|_{L^2(\mu_N)}$. 
Comparing  \eqref{eq:g bound 2} with \eqref{eq:g bound 1}, we see that the error in \eqref{eq:g bound 2} dominates. To finish, we need again to lift the $L^{\infty}(\mu_N)$-bound to a  $L^{\infty}(\mu)$-bound using regularity of the $\psi_i$'s. Notice that  
\begin{align*}
    \|\psi_i^{(N)}\circ T_N-\psi_i\|_{L^{\infty}(\mu)}&=\underset{1\leq j\leq N}{\operatorname{max}}\, \underset{x\in U_j}{\operatorname{sup}}\, |\psi_i^{(N)}\circ T_N(x)-\psi_i(x)| \\
    &\leq \underset{1\leq j \leq N}{\operatorname{max}}\, \underset{x\in U_j}{\operatorname{sup}}\,  \Big(|\psi_i^{(N)}(X_j)-\psi_i(X_j)|+|\psi_i(X_j)-\psi_i(x)|\Big)\\
    &\leq \|\psi_i^{(N)}-\psi_i\|_{L^{\infty}(\mu_N)}+\|\nabla \psi_i\|_{L^{\infty}(\mu)}\rho_N\\
    &\lesssim \lambda_i^{m+1}i^{\frac32}\sqrt{\frac{\rho_N}{\zeta_N}+\zeta_N\sqrt{\lambda_i}},
\end{align*}
where in the last step $\|\nabla \psi_i\|_{L^{\infty}(\mu)}\rho_N$ is a higher order term that we drop.  
\end{proof}

Now we are finally ready to show the $L^{\infty}(\mu)$ convergence of $\Wn$ towards $W^{\M}$.
\begin{proof}[Proof of Theorem \ref{thm:L infty conv}]
Recall that
\begin{align*}
    \Wn&=\sum_{i=1}^{k_N} \left[1+\lambda_i^{(N)}\right]^{-\frac{s}{2}} \xi_i \psi_i^{(N)}\circ T_N, \\
     W^{\M} &=\sum_{i=1}^{\infty} (1+\lambda_i)^{-\frac{s}{2}}\xi_i\psi_i. 
\end{align*}
Consider two intermediate quantities: 
\begin{align*}
    \widetilde{W}_n^{\M}&=\sum_{i=1}^{k_N}(1+\lambda_i)^{-\frac{s}{2}}\xi_i\psi_i^{(N)}\circ T_N, \\
    \widehat{W}_n^{\M}&=\sum_{i=1}^{k_N}(1+\lambda_i)^{-\frac{s}{2}}\xi_i\psi_i.
\end{align*}
By the triangle inequality, 
\begin{align*}
    \mathbb{E}_{\xi}\|\Wn-W^{\M}\|^2_{L^{\infty}(\mu)} 
    &\leq \mathbb{E}_{\xi} \Big(\|\Wn-\widetilde{W}_n^{\M}\|_{L^{\infty}(\mu)}+\|\widetilde{W}_n^{\M}-\widehat{W}_n^{\M}\|_{L^{\infty}(\mu)}+\|\widehat{W}_n^{\M}-W^{\M}\|_{L^{\infty}(\mu)}\Big)^2\\
    &\leq 3\Big(\mathbb{E}_{\xi}\|\Wn-\widetilde{W}_n^{\M}\|^2_{L^{\infty}(\mu)} + \mathbb{E}_{\xi}\|\widetilde{W}_n^{\M}-\widehat{W}_n^{\M}\|^2_{L^{\infty}(\mu)}+ 
    \mathbb{E}_{\xi}\|\widehat{W}_n^{\M}-W^{\M}\|^2_{L^{\infty}(\mu)}\Big),
\end{align*}
so it suffices to bound each term. 
First, by Proposition \ref{thm:efunc gradient bound} and Weyl's law,
\begin{align}
    \mathbb{E}_{\xi}\|\widehat{W}_n^{\M}-W^{\M}\|_{L^{\infty}(\mu)}^2
    &\leq\mathbb{E}_{\xi}\left[\sum_{i=k_N+1}^{\infty}\sum_{j=k_N+1}^{\infty}(1+\lambda_i)^{-\frac{s}{2}}(1+\lambda_j)^{-\frac{s}{2}}|\xi_i||\xi_j| \|\psi_i\|_{L^{\infty}(\mu)}\|\psi_j\|_{L^{\infty}(\mu)}\right]\nonumber\\ 
    &\lesssim \sum_{i=k_N+1}^{\infty}\sum_{j=k_N+1}^{\infty}(1+\lambda_i)^{-\frac{s}{2}}(1+\lambda_j)^{-\frac{s}{2}}\lambda_i^{\frac{m-1}{4}}\lambda_j^{\frac{m-1}{4}}\nonumber\\
    &=\left[\sum_{i=k_N+1}^{\infty}(1+\lambda_i)^{-\frac{s}{2}}\lambda_i^{\frac{m-1}{4}}\right]^2
    \lesssim  \left[\int _{k_N}^{\infty} x^{\frac{-2s+m-1}{2m}}dx\right]^2 
    \lesssim   k_N^{\frac{-2s+3m-1}{m}}. \label{eq:eq1}
\end{align} 
Similarly,
\begin{align*}
    \mathbb{E}_{\xi}\|\Wn-\widetilde{W}_n^{\M}\|^2_{L^{\infty}(\mu)}
    &\lesssim \left[\sum_{i=1}^{k_N}\left|\left[1+\lambda_i^{(N)}\right]^{-\frac{s}{2}}-\Big[1+\lambda_i\Big]^{-\frac{s}{2}}\right|\|\psi_i^{(N)}\circ T_N\|_{L^{\infty}(\mu)}\right]^2. 
\end{align*}
By Lipschitz continuity of $x^{-s/2}$ on $[1,\infty)$ and \eqref{eq:eval rate}, for $i=1,\ldots,k_N,$
\begin{align*}
    \left|\left[1+\lambda_i^{(N)}\right]^{-\frac{s}{2}}-\Big[1+\lambda_i\Big]^{-\frac{s}{2}} \right|
    &\leq \left[(1+\lambda_i^{(N)})\wedge (1+\lambda_i)\right]^{-\frac{s}{2}-1} \left|\lambda_i^{(N)}-\lambda_i\right| \\
    &\lesssim \left[(1+\lambda_i^{(N)})\wedge (1+\lambda_i)\right]^{-\frac{s}{2}-1} \lambda_i \left(\frac{\rho_N}{\zeta_N}+\zeta_N\sqrt{\lambda_i}\right)\\
    &\lesssim \lambda_i^{-\frac{s}{2}}\left(\frac{\rho_N}{\zeta_N}+\zeta_N\sqrt{\lambda_i}\right).
\end{align*}
By \eqref{eq:efun rate}, if we choose $k_N$ so that 
\begin{align}
    k_N^3\left(\frac{\rho_N}{\zeta_N}+\zeta_N\sqrt{\lambda_{k_N}}\right)\ll 1, \label{eq:kn uniform bound}
\end{align}
then we have, for $i=1,\ldots,k_N,$
\begin{align*}
    \|\psi_i^{(N)}\circ T_N\|_{L^{\infty}(\mu)}\leq \|\psi_i^{(N)}\circ T_N-\psi_i\|_{L^{\infty}(\mu)}+\|\psi_i\|_{L^{\infty}(\mu)} \lesssim \lambda_i^{\frac{m-1}{4}}. 
\end{align*}
Therefore 
\begin{align}
    \mathbb{E}_{\xi}\|\Wn-\widetilde{W}_n^{\M}\|^2_{L^{\infty}(\mu)}
    \lesssim  \left[\sum_{i=1}^{k_N}\lambda_i^{-\frac{s}{2}}\left(\frac{\rho_N}{\zeta_N}+\zeta_N\sqrt{\lambda_i}\right)\lambda_i^{\frac{m-1}{4}}\right]^2
    \lesssim \left(\frac{\rho_N}{\zeta_N}+\zeta_N\sqrt{\lambda_{k_N}}\right)^2, \label{eq:eq2}
\end{align}
where we have used that $\lambda_i^{\frac{-2s+m-1}{4}}$ is summable for $s>\frac{3}{2}m-\frac12$. 
Lastly, by Lemma \ref{lemma:L infty rate},
\begin{align}
    \mathbb{E}_{\xi}\|\widetilde{W}_{n}^{\M}-\widehat{W}_n^{\M}\|^2_{L^{\infty}(\mu)} 
    &\lesssim \left[\sum_{i=1}^{k_N}(1+\lambda_i)^{-\frac{s}{2}}\|\psi_i^{(N)}\circ T_N-\psi_i\|_{L^{\infty}(\mu)}\right]^2\nonumber\\
    & \lesssim \left[\sum_{i=1}^{k_N} (1+\lambda_i)^{-\frac{s}{2}}\lambda_i^{m+1}i^{\frac32} \sqrt{ \frac{\rho_N}{\zeta_N}+\zeta_N\sqrt{\lambda_{i}}}\right]^2\nonumber\\
    & \lesssim \left[1\vee k_N^{\frac{-2s+9m+5}{m}}\right]\left(\frac{\rho_N}{\zeta_N}+\zeta_N\right). 
     \label{eq:eq3}
\end{align}
Combining \eqref{eq:eq1}, \eqref{eq:eq2}, \eqref{eq:eq3}, we have
\begin{align}
    \mathbb{E}_{\xi}\|\Wn-W^{\M} \|^2_{L^{\infty}(\mu)}
    &\lesssim  k_N^{\frac{-2s+3m-1}{m}}  + \left[1\vee k_N^{\frac{-2s+9m+5}{m}}\right]\left(\frac{\rho_N}{\zeta_N}+\zeta_N\right).\label{eq:eq4}
\end{align}
Now we will set $\zeta_N$ and $k_N$ based on the dimension $m$ and the smoothness parameter $s$. 

\textbf{Case 1: $s\leq \frac{9}{2}m+\frac{5}{2}$}. Let $\delta>0$ be arbitrary and set
\begin{align*}
    \begin{cases}
    \zeta_N\asymp N^{-\frac{1}{m+4+\delta}}(\log N)^{\frac{p_m}{2}}, \quad  k_N\asymp N^{\frac{1}{(m+4+\delta)(6+\frac{6}{m})}}(\log N)^{-\frac{mp_m}{12m+12}},\quad \quad & m\leq 4,\\
    \zeta_N\asymp N^{-\frac{1}{2m}}(\log N)^{\frac{p_m}{2}},\quad  k_N\asymp N^{\frac{1}{12m+12}}(\log N)^{-\frac{mp_m}{12m+12}}, \quad \quad &m\geq 5.
    \end{cases}
\end{align*}
One can check that the conditions in \eqref{eq:para scale} and \eqref{eq:kn uniform bound} are satisfied and
we get 
\begin{align*}
  \mathbb{E}_{\xi}\|\Wn-W^{\M} \|^2_{L^{\infty}(\mu)} \lesssim
    \begin{cases}
    N^{-\frac{2s-3m+1}{(m+4+\delta)(6m+6)}}(\log N)^{\frac{p_m(2s-3m+1)}{12m+12}} ,\quad \quad &m\leq 4,\\
    N^{-\frac{2s-3m+1}{m(12m+12)}} (\log N)^{\frac{p_m(2s-3m+1)}{12m+12}}, \quad \quad &m\geq 5.
    \end{cases}
\end{align*}

\textbf{Case 2: $s>\frac{9}{2}m+\frac52$}.
Now \eqref{eq:eq4} simplifies to
\begin{align*}
    \mathbb{E}_{\xi}\|\Wn-W^{\M} \|^2_{L^{\infty}(\mu)} \lesssim  k_N^{\frac{-2s+3m-1}{m}}+\left(\frac{\rho_N}{\zeta_N}+\zeta_N\right).
\end{align*}
Therefore by setting
\begin{align*}
    \begin{cases}
    \zeta_N\asymp N^{-\frac{1}{m+4+\delta}}(\log N)^{\frac{p_m}{2}},\quad  k_N\asymp N^{\frac{m}{(m+4+\delta)(2s-3m+1)}}(\log N)^{-\frac{mp_m}{4s-6m+2}},\quad \quad  & m\leq 4, \\
    \zeta_N\asymp N^{-\frac{1}{2m}}(\log N)^{\frac{p_m}{2}},\quad  k_N\asymp N^{\frac{1}{4s-6m+2}}(\log N)^{-\frac{mp_m}{4s-6m+2}},\quad \quad & m\geq 5,
    \end{cases}
\end{align*}
we have 
\begin{align*}
    \mathbb{E}_{\xi}\|\Wn-W^{\M} \|^2_{L^{\infty}(\mu)} \lesssim
    \begin{cases}
    N^{-\frac{1}{m+4+\delta}}(\log N)^{\frac{p_m}{2}}, \quad \quad &m\leq 4,\\
    N^{-\frac{1}{2m}}(\log N)^{\frac{p_m}{2}}, \quad \quad &m\geq 5.
    \end{cases}
\end{align*}
Again one can check that the conditions in \eqref{eq:para scale} and \eqref{eq:kn uniform bound} are satisfied. This finishes the proof.

\end{proof}

\subsection{Putting Everything Together}\label{sec:combine}
Now we are ready to prove Theorem \ref{thm:Main Thm}. 
\begin{proof}[Proof of Theorem \ref{thm:Main Thm}]
By Subsection \ref{sec:reduction}, it suffices to show that 
\begin{align*}
    \mathbb{E}_{X}\mathbb{E}_{f_0}\Pim\left(f\in\mathcal{I}(L^{\infty}(\mu_{N_n})):\|f-f_0\|_n\geq M_n\eps_n\,|\, \X_{N_n}, \Y_n\right) \xrightarrow{n\rightarrow\infty} 0,
\end{align*}
with the said $\eps_n$. 
Let $A_n$ be the high probability event in Corollary \ref{cor:L infty conv}. Denote $F_n(\X_{N_n},\Y_n):= \Pim(f\in\mathcal{I}(L^{\infty}(\mu_{N_n})):\|f-f_0\|_n\geq M_n\eps_n\, |\, \Y_n)$. We have 
\begin{align*}
    \mathbb{E}_X\mathbb{E}_{f_0} F_n(\X_{N_n},\Y_n) = \mathbb{E}_{X}[\mathbb{E}_{f_0}F_n(\X_{N_n},\Y_n)]\mathbf{1}_{A_n} + \mathbb{E}_{X}[\mathbb{E}_{f_0}F_n(\X_{N_n},\Y_n)]\mathbf{1}_{A_n^c}.
\end{align*}
Since $F_n(\X_{N_n},\Y_n)\leq 1$, the second term is upper bounded by $\mathbb{P}_X(A_n^c)\rightarrow 0$. It then suffices to show $\mathbb{E}_{f_0}F_n(\X_{N_n}, \Y_n)\rightarrow 0$ in the event of $A_n$. By Corollary \ref{cor:L infty conv} we can have $10\mathbb{E}_{\xi}\|\Wn-W^{\M}\|^2_{L^{\infty}(\mu)}\leq n^{-1}$ if we set a large enough proportion constant for $N_n$. This together with Theorem \ref{thm:cont rate} and \cite{van2008rates}[Theorem 2.2] implies that there exists sets  $B_n\subset \mathcal{I}( L^{\infty}(\mu_{N_n}))$ (by taking the intersection of the sets $B_n$ in Theorem \ref{thm:cont rate} and $\mathcal{I}( L^{\infty}(\mu_{N_n}))$)  so that, for the same $\eps_n$ in the theorem statement,
\begin{align*}
    \log N(6\eps_n,B_n,\|\cdot\|_{L^{\infty}(\mu)})&\leq  24Cn\eps_n^2, \\
    \pim(B_n^c)&\leq e^{-4Cn\eps_n^2},\\
    \pim(\|w-\Phi^{-1}(f_0)\|_{L^{\infty}(\mu)}<4\eps_n)&\geq e^{-4n\eps_n^2},
\end{align*}
where $\pim=\mathcal{L}(\Wn)$. 

\textbf{Case 1: Regression}. In the regression case, since $\Phi$ is the identity, we have $\Pim=\pim$ and the above three conditions are true with $\pim$ replaced by $\Pim$. Furthermore, since $\|\cdot\|_n$ is upper bounded by $\|\cdot\|_{L^{\infty}(\mu)}$, the above conditions remain true for the empirical norm. This together with the general results in \cite{ghosal2017fundamentals}[Section 8.3.2] proves the assertion. 

\textbf{Case 2: Classification}. In the classification setting, by \cite{ghosal2017fundamentals}[Lemma 2.8] the average Kullback-Leibler divergence and variance in \eqref{eq:KL neighbor} between the densities after applying the link function $\Phi$ are upper bounded by a multiple of the empirical norm. In particular if we set $\mathcal{B}_n=\Phi(B_n)$, then the conditions in \eqref{eq:entropy}, \eqref{eq:complement} and \eqref{eq:prior mass} hold, for a possibly different proportion constant for $\eps_n$. Moreover, as discussed before, the root average square Hellinger distance $d_n$ in \eqref{eq:RASH} is equivalent to the empirical norm. Hence the result follows by \cite{ghosal2007convergence}[Theorem 4].
\end{proof}

\section{Discussion}\label{sec:discussion}
In this paper we have analyzed graph-based SSL using recent results on spectral convergence of graph Laplacians and standard Bayesian nonparametrics techniques.  We show that, for a suitable choice of prior constructed with sufficiently many unlabeled data, the posterior contracts around the truth at a rate that is minimax optimal up to logarithmic factor. Our theory applies to both regression and classification. 

We have assumed throughout that the $X_i$'s are uniformly distributed on $\M.$ Our results can be generalized to nonuniform positive density $q$ with respect to the volume form. In such a case, the continuum field is the Gaussian measure $\mathcal{N}(0, (I-\Delta_{q}))^{-s}$ where  $-\Delta_q:=-\frac{1}{q}\text{div} (q^2 \nabla )$ is a weighted Laplacian-Beltrami operator, for which spectral convergence results can be found in \cite{trillos2019error}. Since we do not have an explicit dependence of the proportion constant on the index $i$ as in Proposition \ref{thm:specRate}, we have chosen to present our result in the uniform case. However, similar conclusions should be expected to hold in the general case.

The Bayesian methodology we analyzed is inspired by popular existing graph-based optimization methods for SSL \citep{belkin2004regularization,zhu2005semi}. In order to achieve optimal contraction, the prior smoothness parameter needs to match the regularity of the truth, which is rarely available in applications. An important research direction stemming from our work is the development and analysis of adaptive Bayesian SSL methodologies that can achieve optimal contraction without a priori smoothness information.  We expect that the existing results on adaptive estimation on manifolds \citep{castillo2014thomas} and on large graphs \citep{kirichenko2017estimating} will be an important stepping stone in this direction.

\section*{Acknowledgment}
Both authors are thankful for the support of NSF and NGA through the grant DMS-2027056. 
The work of DSA was also partially supported by the NSF Grant DMS-1912818/1912802. 
\bibliographystyle{abbrvnat} 

\bibliography{reference}
\end{document}